\documentclass{article}

\usepackage[final]{neurips_2025}

\usepackage[utf8]{inputenc} 
\usepackage[T1]{fontenc}    
\usepackage{url}            
\usepackage{booktabs}       
\usepackage{amsfonts}       
\usepackage{nicefrac}       
\usepackage{microtype}      
\usepackage{xcolor}         
\definecolor{darkgreen}{rgb}{0.0, 0.5, 0.0}
\definecolor{darkblue}{rgb}{0.0, 0.0, 0.55}
\usepackage[colorlinks=true, linkcolor=darkblue, citecolor=darkblue]{hyperref}  

\usepackage{graphicx}
\usepackage{float}
\usepackage{mathtools}
\usepackage{amssymb, amsmath, latexsym, bm}
\usepackage{amsthm}
\usepackage{nicefrac}
\usepackage{hhline}
\usepackage{multirow}
\usepackage{pifont}
\usepackage[customcolors]{hf-tikz}

\usepackage[colorinlistoftodos,bordercolor=orange,backgroundcolor=orange!20,linecolor=orange,textsize=scriptsize]{todonotes}
 \usepackage{algorithm}
 \usepackage{algpseudocode}
\usepackage{caption}
\usepackage{subcaption}
\usepackage{enumitem}
\usepackage{natbib}
\usepackage{sidecap}

\usepackage[flushleft]{threeparttable} 

\usepackage{caption}
\usepackage{multirow}
\usepackage{colortbl}
\definecolor{bgcolor}{rgb}{0.8,1,1}

\usepackage{pifont}
\definecolor{PineGreen}{RGB}{0,110,51}
\definecolor{BrickRed}{RGB}{143,20,2}
\newcommand{\cmark}{{\color{PineGreen}\ding{51}}}%
\newcommand{\xmark}{{\color{BrickRed}\ding{55}}}%

\newcommand{\eqdef}{\overset{\text{def}}{=}}

\newcommand{\cD}{{\cal D}}
\newcommand{\cE}{{\cal E}}
\newcommand{\cF}{{\cal F}}

\newcommand{\cO}{{\cal O}}

\newcommand{\cU}{{\cal U}}

\newcommand{\sign}{\mathrm{sign}}

\newcommand{\EE}{\mathbf{E}}

\def\R{\mathbb{R}}
\newcommand{\E}{{\mathbb E}}

\def\R{\mathbb R}
\def\E{\mathbb E}
\def\EE{\mathbb E}
\def\PP{\mathbb P}

\def\la{\langle}
\def\ra{\rangle}

\newcommand{\norm}[1]{\left\lVert#1\right\rVert}
\newcommand{\Exp}[2][]{
	\ifthenelse{\equal{#1}{}}{\E \left[ #2 \right]}{\E \left[ #2 \, \middle| \, #1 \right]}
}

\newcommand{\sm}{L}

\newcommand{\rb}[1]{\left(#1\right)}

\newtheorem{lemma}{Lemma}
\newtheorem{theorem}{Theorem}
\newtheorem{definition}{Definition}

\newtheorem{assumption}{Assumption}
\newtheorem{corollary}{Corollary}

\newcommand{\circledOne}{\text{\ding{172}}}
\newcommand{\circledTwo}{\text{\ding{173}}}
\newcommand{\circledThree}{\text{\ding{174}}}
\newcommand{\circledFour}{\text{\ding{175}}}
\newcommand{\circledFive}{\text{\ding{176}}}
\newcommand{\circledSix}{\text{\ding{177}}}
\newcommand{\circledSeven}{\text{\ding{178}}}
\newcommand{\circledEight}{\text{\ding{179}}}

\definecolor{linen}{HTML}{FAF0E6} 

\usepackage{comment}

\usepackage{ifthen}
\newboolean{showcomments}
\setboolean{showcomments}{true} 

\newcommand{\ilyas}[1]{\ifthenelse{\boolean{showcomments}}{\todo[inline]{\textbf{Ilyas: }#1}}{}}
\newcommand{\abdurakhmon}[1]{\ifthenelse{\boolean{showcomments}}{\todo[inline]{\textbf{Abdurakhmon: }#1}}{}}

\title{Second-order Optimization under Heavy-Tailed Noise: Hessian Clipping and Sample Complexity Limits}

\author{Abdurakhmon Sadiev\thanks{Corresponding author: \texttt{abdurakhmon.sadiev@kaust.edu.sa}}\\ KAUST\thanks{King Abdullah University of Science and Technology, Thuwal, Saudi Arabia.}\\
Center of Excellence for Generative AI \And Peter Richtárik \\ KAUST\\ Center of Excellence for Generative AI \And Ilyas Fatkhullin \\ETH Zurich, ETH AI Center, \\ Georgia Institute of Technology  }

\begin{document}

\maketitle

\begin{abstract}
  Heavy-tailed noise is pervasive in modern machine learning applications, arising from data heterogeneity, outliers, and non-stationary stochastic environments. While second-order methods can significantly accelerate convergence in light-tailed or bounded-noise settings, such algorithms are often brittle and lack guarantees under heavy-tailed noise---precisely the regimes where robustness is most critical. In this work, we take a first step toward a theoretical understanding of second-order optimization under heavy-tailed noise. We consider a setting where stochastic gradients and Hessians have only bounded $p$-th moments, for some $p\in (1,2]$, and establish tight lower bounds on the sample complexity of any second-order method. We then develop a variant of normalized stochastic gradient descent that leverages second-order information and provably matches these lower bounds. To address the instability caused by large deviations, we introduce a novel algorithm based on gradient and Hessian clipping, and prove high-probability upper bounds that nearly match the fundamental limits. Our results provide the first comprehensive sample complexity characterization for second-order optimization under heavy-tailed noise. This positions Hessian clipping as a robust and theoretically sound strategy for second-order algorithm design in heavy-tailed regimes.
\end{abstract}

\section{Introduction}

We consider the stochastic optimization problem
\begin{eqnarray}\label{eq:main_problem}
\min_{x \in \mathbb{R}^d} F(x), \qquad F(x) := \mathbb{E}_{\xi \sim \mathcal{D}}[f(x, \xi)],
\end{eqnarray}
with $F$ a smooth (but potentially nonconvex) objective, and $\xi$ a random variable drawn from an unknown distribution. While first-order methods are widely used in practice due to their simplicity and scalability, the application of second-order methods remains limited in stochastic settings. This is in stark contrast to their strong theoretical properties: second-order methods such as Newton’s method \citep{bennett1916newton,kantorovich1948functional}, cubic regularization \citep{nesterov2006cubic,nesterov2008accelerating}, their variants for finite-sum minimization \citep{kovalev2019stochastic}, quasi-Newton \citep{dennis1977quasi}, and adaptive second-order algorithms \citep{doikov2024super} can offer provably faster convergence rates. Furthermore, complexity-theoretic results demonstrate that second-order methods can outperform first-order ones in noiseless (or finite sum) settings, often with  moderate additional computational cost \citep{agarwal2018lower,arjevani2019oracle,arjevani2020second} only.

\paragraph{Second-order methods are highly sensitive to noise.}
First-order stochastic optimization (FOSO) methods--such as stochastic gradient descent (SGD) and its adaptive variants--are well understood, both algorithmically and from a complexity--theoretic perspective. These methods enjoy optimal convergence guarantees under a broad class of noise models, including bounded variance and infinite variance regimes \citep{gower2019sgd,khaled2020better,yang2023two_sides,Wang_SGD_inf_var_21,fatkhullin2025can}. In contrast, the development of second-order stochastic optimization (SOSO) methods has been significantly hampered by their extreme sensitivity to noise, particularly in Hessian estimation. Existing second-order methods suffer from both practical instability and overly restrictive theoretical assumptions. For instance, stochastic extensions of cubic regularization \citep{ghadimi2017second,tripuraneni2018stochastic} and its momentum variants \citep{chayti2024improving} require bounded noise assumptions, which are rarely satisfied in modern large-scale learning tasks. Other frameworks--such as trust-region methods \citep{arjevani2020second}, recursive momentum schemes \citep{BetterSGDUsingSOM_Tran_2021}, and extrapolation-based approaches \citep{antonakopoulos2022extra,agafonov2023advancing}--relax this to bounded variance, but still fall short of handling more realistic heavy-tailed noise distributions. To the best of our knowledge, no existing SOSO method can operate reliably under noise conditions with unbounded variance, highlighting a fundamental gap in the current landscape of stochastic optimization.

\paragraph{Heavy-tailed noise in gradients and Hessians.} In recent years, the first-order optimization literature has increasingly moved beyond the bounded variance assumption, adopting more general noise models that better reflect empirical observations in modern applications. A particularly influential framework is the \emph{bounded central moment} condition ($p$-BCM), which assumes
\[
\mathbb{E}\left[\|\nabla f(x, \xi) - \nabla F(x)\|^p\right] \le \sigma^p, \qquad \text{for some } p \in (1, 2],
\]
thereby allowing for heavy-tailed noise with unbounded variance. This model aligns well with empirical findings in deep learning and reinforcement learning (RL), where heavy-tailed gradient noise is often observed \citep{garg2021proximal,zhang2020adaptive,ahn2023linear,simsekli2019tail,battash2024revisiting}. Algorithms designed for robustness under $p$-BCM noise, such as those employing gradient clipping or normalization, have demonstrated both strong empirical performance and rigorous convergence guarantees \citep{zhang2020adaptive,Hubler2024clip_to_norm}. These developments suggest that heavy-tailed noise models are not only practically relevant, but also theoretically tractable, providing a principled foundation for algorithm design. Given that second-order methods are even more sensitive to noise than first-order ones, this naturally motivates the extension of such robustness principles to the second-order setting. In this work, we take this step and aim to
\begin{center}
\textit{develop a comprehensive theory of second-order stochastic optimization under heavy-tailed noise.}
\end{center}
Specifically, we assume access to unbiased stochastic gradients and Hessian-vector products, each satisfying a $p$-BCM-type condition:
\[
\mathbb{E}\left[\|\nabla^2 f(x, \xi) - \nabla^2 F(x)\|_{\text{op}}^p\right] \le \sigma_h^p, \qquad \text{for some } p \in (1, 2].
\]
Our goal is to characterize the fundamental performance limits and develop practical algorithms for minimizing $F(x)$ in this setting, with a focus on obtaining guarantees for finding points with small gradient norm $\|\nabla F(x)\| \leq \varepsilon$, either in expectation or with high probability.\footnote{It would be interesting to extend our results to finding second-order stationary points \citep{nesterov2006cubic,arjevani2020second}, but we leave this investigation for future work.}

\paragraph{From gradient to Hessian clipping.} Gradient clipping has become a standard tool in modern machine learning, particularly due to its empirical success in stabilizing training under heavy-tailed noise and ill-conditioned objectives \citep{pascanu2013difficulty,schulman-et-al17ppo}. Theoretically, it has been shown to offer robustness under relaxed smoothness and moment assumptions \citep{polyak1979adaptive,jakovetic2023nonlinear}, and enable high-probability guarantees with logarithmic (sub-Gaussian) dependence on the failure probability. The latter property of gradient clipping is in stark contrast to the classical linear algorithms, e.g., SGD, Momentum SGD, for which high probability lower bounds were recently established \citep{fatkhullin2025can}. These properties are well-documented across both convex \citep{nazin2019algorithms,gorbunov2020stochastic,davis2021low,gorbunov2023high,liu2023stochastic,gorbunov2024high,puchkin2024breaking,NonlinearHP2023Armacki} and nonconvex \citep{sadiev2023high,nguyen2023improved,cutkosky2021high,Hubler2024clip_to_norm} regimes. Crucially, high-probability results are both theoretically and practically appealing as they guarantee the performance of individual runs rather than for the average-case behavior.

Despite these advances, the robustness enabled by gradient clipping remains largely confined to first-order methods. In the second-order setting, where both gradients and Hessians may be corrupted by heavy-tailed noise, comparable algorithmic tools are conspicuously lacking. The core difficulty is not just technical but conceptual: existing SOSO algorithms do not include mechanisms to suppress the influence of extreme noise in Hessian estimates. This leads us to a fundamental question for the design of robust second-order stochastic methods:
\begin{center}
\textit{Can we develop second-order algorithms with high-probability convergence guarantees under simultaneous heavy-tailed noise in both gradients and Hessians?}
\end{center}

\begin{figure}[ht]
    \begin{minipage}[t]{0.48\linewidth}
        \centering
        \includegraphics[width=\linewidth]{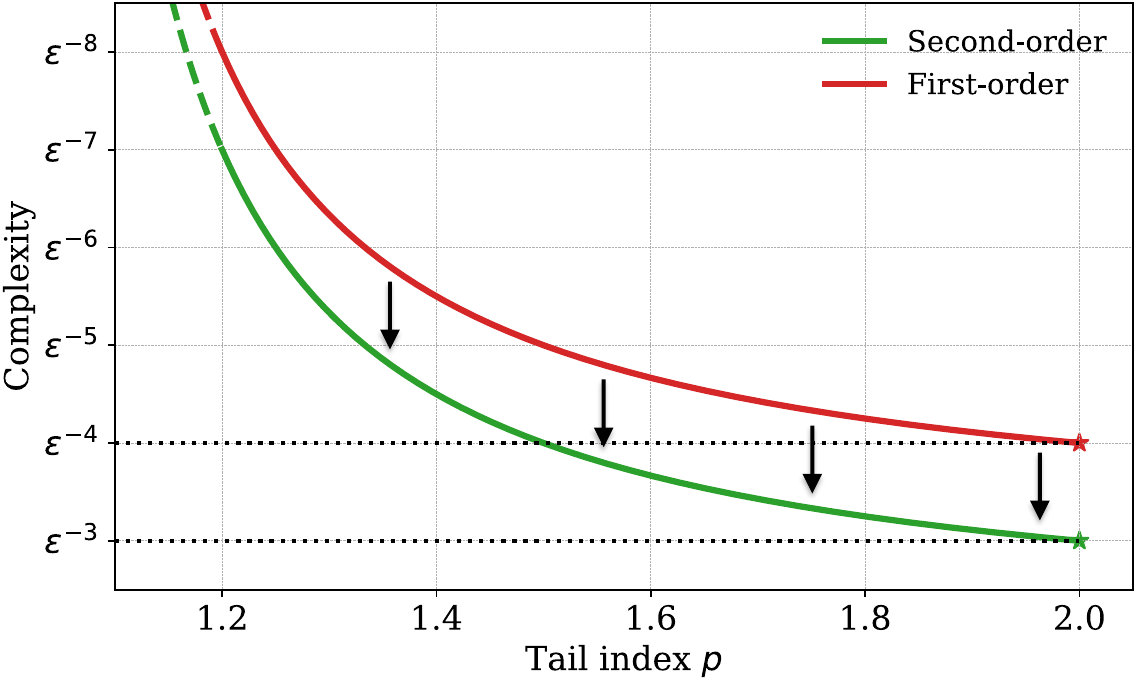}
        \caption{Sample complexity comparison for FOSO and SOSO depending on the tail index $p$. Each line corresponds to the leading term in the sample complexity for each class of algorithms. These leading terms match in upper and lower bounds, so this characterization is exact. We establish the characterization along the entire green line complementing  prior work \citep{arjevani2023lower} for $p=2$. 
        }
        \label{fig:complexities}
    \end{minipage}%
    \hfill
    \begin{minipage}[t]{0.47\linewidth}
        \centering
        \includegraphics[width=\linewidth]{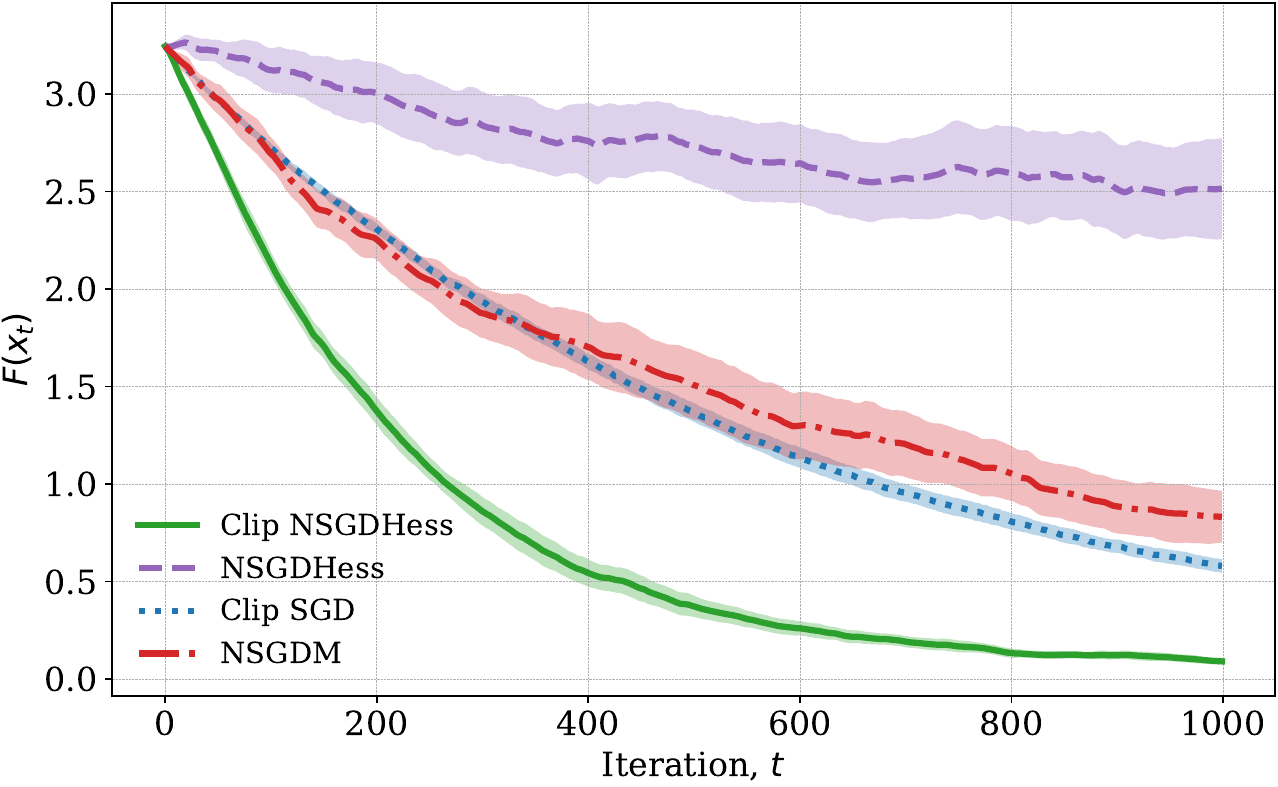}
        \caption{Performance of algorithms on a simple problem, $F(x) = 0.5 \norm{x}^2$, $d = 10$ with synthetic noise generated from a two-sided Pareto distribution with tail index $p=1.1$. We observe that algorithms without clipping, NSGDM and NSGDHess, suffer significantly from noise. This motivates our more in-depth study involving gradient and Hessian clipping for high probability convergence.}
        \label{fig:compare-algs}
    \end{minipage}
\end{figure}

\paragraph{Contributions:}
\begin{itemize}
    \item \textbf{Tight lower bound.} We establish a minimax lower bound on the sample complexity of any SOSO algorithm under $p$-BCM noise model:
\(
\Omega\left( \frac{\Delta \sigma_h}{\varepsilon^2} \left( \frac{\sigma}{\varepsilon} \right)^{\frac{1}{p - 1}} \right),
\)
where $\Delta$ denotes the initial suboptimality, and $\sigma$, $\sigma_h$ denote the scale of noise in gradients and Hessians, respectively. This bound improves over the best known complexity for first-order methods \citep{zhang2020adaptive,Hubler2024clip_to_norm} by a factor of $1/\varepsilon$ uniformly for all $p\in (1,2]$, demonstrating that second-order methods can yield provable advantages even in the heavy-tailed regime; see Figure~\ref{fig:complexities}. Moreover, our result implies that increasing the order of the algorithm (e.g., using higher-order derivatives) cannot improve the rate under the same noise assumptions significantly.

\item \textbf{Optimal algorithm.} We develop a second-order stochastic algorithm that achieves the lower bound (up to constants) in the regime $L \leq \sigma_h$, where $L$ denotes the Lipschitz constant of the gradient. The algorithm attains the sample complexity 
\(
\mathcal{O}\left( \frac{\Delta (L + \sigma_h)}{\varepsilon^2} \left( \frac{\sigma}{\varepsilon} \right)^{\frac{1}{p - 1}} \right),
\)
and is, to the best of our knowledge, the first second-order method that provably works under $p$-BCM noise. Notably, our method does not require second-order smoothness of the objective. This result holds even in the classical setting $p = 2$, highlighting the broader implications of the approach.

\item \textbf{High-probability convergence via Hessian clipping.} We propose a clipped variant of our algorithm that incorporates both gradient and \emph{Hessian clipping}--the latter introduced here for the first time. We show that this variant achieves near-optimal sample complexity with high probability, incurring only a poly-logarithmic overhead in the failure probability. This significantly extends the robustness of SOSO methods, enabling strong guarantees not only in expectation but also with high confidence.

\end{itemize}

Our results are primarily theoretical and validated through synthetic experiments with controllable heavy-tailed noise. While one component of our algorithm (in its unclipped form) has previously been applied in reinforcement learning settings \citep{Fatkhullin_SPGM_FND_2023}, its broader practical deployment remains an open challenge. In particular, several questions must be addressed to make these methods viable in real-world applications: How can we reduce hyperparameter tuning overhead? Are there default configurations that are robust across problem classes? Is clipping truly necessary, or can it be systematically replaced by normalization techniques, as recent work suggests in the first-order context \citep{Hubler2024clip_to_norm}? We view our contributions as a foundational step towards a principled understanding of second-order methods under realistic, heavy-tailed noise conditions, and hope they spur further theoretical and algorithmic advances in this direction.

\section{Notations and Assumptions}
We use the common notation for natural numbers $\mathbb{N} = \left\{0, 1, \dots\right\}$, and let $[n] = \left\{1, 2, \dots, n\right\}$. Throughout this paper, $d\in\mathbb{N}_{\geq1}$ denotes the dimension of the optimization problem \eqref{eq:main_problem}. We use $\cO\rb{\cdot}, \Omega\rb{\cdot}$ for sample complexity notations, which preserve the dominating term in the target accuracy $\varepsilon>0$ along with the multiplicative constants such as initial suboptimality $\Delta$, smoothness constants $L_r$, $r \geq 1$, the variance of higher-order stochastic oracles $\sigma_r$, $r\geq 1$. When we write $\bar\cO\rb{\cdot}$, we suppress the dependence on all parameters except for accuracy $\varepsilon$ in the dominating term.
As in \cite{carmon2020lower}, for any $q$-th order tensor $T \in \R^{\otimes^q d}$, we define the support of $T$ as $\operatorname{supp}\{T\} \eqdef \left\{i \in [d]~|~ T_i \neq 0\right\}$, where $T_i$ is the $(q-1)$-th order tensor, given by $[T_i]_{j_1,\ldots,j_{q-1}} =T_{i,j_1,\ldots,j_{q-1}} $.

We make the following assumptions throughout the paper.  
\begin{assumption}[Lower Boundedness]\label{assum:lower_bounded}
	The objective function $F$ is lower bounded by $F^* > -\infty$.
\end{assumption}

\begin{assumption}[$\sm$-smoothness]\label{assum:l_smooth}
	The objective function $F$ is $L$-smooth, i.e.,\ $F$ is differentiable and for all $x,y\in\R^d$, we have $\norm{\nabla F(x) -\nabla F(y)} \leq L \norm{x-y}$.
\end{assumption}

\begin{assumption}[$p$-BCM]\label{ass:p_BCM}
    We have access to stochastic gradients $\nabla f(x, \xi)$ such that $\Exp{\nabla f(x, \xi)} = \nabla F(x)$. Moreover, there exist $p \in (1, 2]$ and $\sigma > 0$ such that 
    $$
    \Exp{\norm{\nabla f(x, \xi) - \nabla F(x)}^{p}} \leq \sigma^p  \qquad \text{for all } x \in \R^d .
    $$
\end{assumption}

\begin{assumption}[$p$-BCM for Hessian]\label{ass:p_BCM_hess}
    The function $F$ is twice differentiable and we additionally have access to stochastic Hessian-vector products $\nabla^2 f(x, \xi) \cdot v$ for any $v\in \R^d$ such that $\Exp{\nabla^2 f(x, \xi)} = \nabla^2 F(x)$. Moreover, there exist $p \in (1, 2]$ and $\sigma_h > 0$ such that
    $$
    \Exp{\norm{\nabla^2 f(x, \xi) - \nabla^2 F(x)}_{\text{op}}^{p}} \leq \sigma_h^p  \qquad \text{for all } x \in \R^d .
    $$
\end{assumption}

\section{Lower Complexity Bounds for SOSO under $p$-BCM}
\label{sec:lower_bounds}

To establish lower bounds, we build on a technique developed in a series of works \citep{arjevani2023lower, carmon2020lower, arjevani2020second}, which introduced a \textit{worst-case} non-convex function exhibiting the zero-chain property: starting from the origin, each algorithmic iteration reveals information about a single coordinate only. Leveraging this framework, \citet{zhang2020adaptive} derived the first lower bounds for FOSO under heavy-tailed noise, assuming the $p$-th moment of the stochastic gradient is bounded, i.e., $\EE\left[\norm{\nabla f(x,\xi)}^p\right] \leq G^p$, which differs from Assumption~\ref{ass:p_BCM}.

In our work, we assume access to a \( q \)-th order stochastic oracle, meaning that each query yields stochastic approximations of derivatives of orders \( 1 \) to \( q \) at a given point \( x \). This oracle model was formally introduced by \citet{arjevani2020second}. Building on their proof technique, we derive new lower bounds under the assumption that each stochastic derivative has a bounded \( p \)-th central moment.

\textbf{Function class.} We consider the class of \( q \)-times differentiable functions (\( q \geq 1 \)) satisfying two key properties: (i) functional boundedness, i.e., \( F(0) - F^* \leq \Delta \); and (ii) Lipschitz continuity of each derivative up to order \( q \), i.e., for all \( x, y \in \R^d \) and \( r \in [q] \), 
\[
\|\nabla^{r}F(x) - \nabla^{r}F(y)\|_{\text{op}} \leq L_{r}\|x - y\|.
\]
We denote this function class by \( \cF(\Delta, L_{1:q}) \), where \( L_{1:q} = \{L_1, L_2, \dots, L_q\} \). Note that the constant \( L \) from Assumption~\ref{assum:l_smooth} corresponds to \( L_1 \) in this notation.

\textbf{Oracle class.} We now define the oracle model. For a fixed function \( F \in \cF(\Delta, L_{1:q}) \), a stochastic \( q \)-th order oracle is defined as follows: for any \( x \in \R^d \) and \( \xi \sim \cD \),
\begin{equation*}
    \text{O}^q_f(x,\xi) \eqdef \left( f(x,\xi), \nabla f(x,\xi), \nabla^2 f(x,\xi), \dots, \nabla^q f(x,\xi) \right),
\end{equation*}
where each component satisfies the unbiasedness conditions:
\(\EE_{\xi\sim \cD}[f(x,\xi)] = F(x), \quad \text{and} \quad \EE_{\xi\sim\cD}[\nabla^r f(x,\xi)] = \nabla^r F(x), \quad \forall~r \in [q].\) We define the oracle class \( \cO_q(F, \sigma_{1:q}) \) as the set of all such stochastic oracles for which the \( p \)-th moment of the estimation error (with \( p \in [1,2) \)) satisfies
\begin{equation*}
    \EE_{\xi\sim \cD}\left[\left\| \nabla^r f(x,\xi) - \nabla^r F(x) \right\|_{\text{op}}^p\right] \leq \sigma_r^p, \quad \forall~r \in [q],~~ \forall~x \in \R^d.
\end{equation*}
For \( r = 1 \), this recovers \( \sigma_1 = \sigma \) from Assumption~\ref{ass:p_BCM}, and for \( r = 2 \), we similarly have \( \sigma_2 = \sigma_h \) from Assumption~\ref{ass:p_BCM_hess}.

\textbf{Algorithm class.} To prove lower bounds, we will work with a class of \textit{zero-respecting} algorithms. The formal definition is provided below. The reason why this class of algorithms is interesting for us is based on its property: at each iteration of such an algorithm, at most one new coordinate can be  revealed. 

\begin{definition}
    We call a stochastic $q$-th order algorithm $A$ zero-respecting if for any function $F $ and any $p$-th order oracle $\text{O}^q_{F}$, the iterates $\{x^t\}_{t\in \mathbb{N}}$ generated by $A$ via querying $\text{O}^q_{F}$ satisfy the following property: \(
        \text{supp}\left\{x^{t}\right\} \subset \bigcup_{i < t} \text{supp}\left\{\text{O}^q_F (x^{i}, \xi^{i})\right\}, \quad \forall~ t \in \mathbb{N} \) 
    with probability one with respect to the randomness of the algorithm and the realizations of $\{\xi^{t}\}_{t \in \mathbb{N}}$.
\end{definition}

\begin{theorem}
\label{thm:lower_bounds}
Let $q\in \mathbb{N}_{\geq 1}$,  and let $\Delta >0, L_{1:q}\eqdef (L_1,\dots, L_q)$, $\sigma_{1:q} \eqdef (\sigma_1,\dots, \sigma_q)  $  and $\varepsilon \leq \cO(\sigma_1)$. Then, there exists $F \in \cF(\Delta, L_{1:q})$ and a corresponding noisy oracle $\text{O}^q_F \in \cO(F, \sigma_{1:q})$  such that for any $q$-th order zero-respecting algorithm, the number of oracle queries required to find an $\varepsilon$-stationary point (with constant probability) is  lower bounded  by 
\begin{equation}
    \Omega(1)\cdot\frac{\Delta}{\varepsilon}\left(\frac{\sigma_1}{\varepsilon}\right)^{\frac{p}{p-1}}  \min\left\{ \min_{r \in \{2,\dots, q\}}\left(\frac{\sigma_{r}}{\sigma_1}\right)^{\frac{1}{r-1}}, \min_{r' \in [q]}\left(\frac{L_{r'}}{\varepsilon}\right)^{\frac{1}{r'}}\right\}.
\end{equation}
Moreover, this lower bounds is realized by a construction of dimension $\Theta\left(\frac{\Delta}{\varepsilon}  \min\left\{ \min_{r \in \{2,\dots, q\}}\left(\frac{\sigma_{r}}{\sigma_1}\right)^{\frac{1}{r-1}}, \min_{r' \in [q]}\left(\frac{L_{r'}}{\varepsilon}\right)^{\frac{1}{r'}}\right\}\right).$
\end{theorem}
The proof of Theorem~\ref{thm:lower_bounds} is deferred to Appendix~\ref{app:missing_proofs_lower_bounds}. In Theorem~\ref{thm:lower_bounds} we provide lower bounds for any $q \in \mathbb{N}_{\geq 1}$, which corresponds to the order of an oracle used in an optimization algorithm. 
When the oracle order satisfies $q \ge 2$, the algorithm gains access to stochastic curvature information, enabling more accurate prediction of the function or gradient behavior between consecutive iterates. This richer information can potentially lead to tighter worst-case complexity limits compared to first-order methods.
Since the main focus of our work is on the second-order setting, we now discuss the complexity result from Theorem~\ref{thm:lower_bounds} for $q = 2$. For second-order methods, $q=2$, under Assumptions~\ref{assum:l_smooth},~\ref{ass:p_BCM},~\ref{ass:p_BCM_hess} and, additionally, (possibly) assuming the Hessian is Lipschitz continuous with constant $L_h$, we have 
\begin{equation}
\label{eq:lower_bound_second_order}
    \Omega(1)\cdot \min\left\{\frac{\Delta \sigma_h}{\varepsilon^2}\left(\frac{\sigma}{\varepsilon}\right)^{\frac{1}{p-1}}, \frac{\Delta L \sigma}{\varepsilon^{3}} \left(\frac{\sigma}{\varepsilon}\right)^{\frac{1}{p-1}}, \frac{\Delta \sqrt{L_h}\sigma}{\varepsilon^{\nicefrac{5}{2}}} \left(\frac{\sigma}{\varepsilon}\right)^{\frac{1}{p-1}}\right\},
\end{equation}
where the parameters satisfy $\sigma = \sigma_1$, $\sigma_h = \sigma_2$, $L = L_1$, $L_2 = L_h$. We will now discuss each term separately, in different regimes.

\textbf{First-order optimization with Lipschitz gradient.} When we do not have access to the second-order information and the objective function has Lipschitz continuous gradient only, but not any higher derivative, our lower bounds \eqref{eq:lower_bound_second_order} reduces to the second term:
\begin{equation*}
    \Omega\left(\frac{\Delta L \sigma}{\varepsilon^{3}} \left(\frac{\sigma}{\varepsilon}\right)^{\frac{1}{p-1}}\right).
\end{equation*}
This bound exactly matches (up to a numerical constant factor) the previously known lower bounds for first-order stochastic optimization under $p$-BCM \citep{zhang2020adaptive}, and the corresponding upper bound achieved by normalized SGD \citep{Hubler2024clip_to_norm}. Several other algorithms, like Clip-SGD and NSGD with momentum and clipping also nearly match this complexity lower bound, up to additional polylogarithmic terms in $1/\varepsilon$ and polynomial terms in $\delta$, $L$ and $\sigma$ \citep{zhang2020adaptive,cutkosky2021high,nguyen2023improved}.

\paragraph{First-order optimization under higher-order smoothness.} The second case worth discussing is when we (still) do not have access to stochastic Hessian, but we know that the objective function has Lipschitz continuous second-order derivatives. Then our lower bound \eqref{eq:lower_bound_second_order} becomes 
\begin{equation*}
    \Omega\left(\min\left\{\frac{\Delta L \sigma}{\varepsilon^{3}} \left(\frac{\sigma}{\varepsilon}\right)^{\frac{1}{p-1}}, \frac{\Delta \sqrt{L_h}\sigma}{\varepsilon^{\nicefrac{5}{2}}} \left(\frac{\sigma}{\varepsilon}\right)^{\frac{1}{p-1}} \right\}\right).
\end{equation*}
The last term in the above bound nearly matches with the complexity of the method called NIGT with clipping, by \cite{cutkosky2021high}, in terms of its dependence on $1/\varepsilon$, up to logarithmic factors. Unfortunately, there is still a small discrepancy with the upper bound in the dependence on other parameters and with the fact that \citet{cutkosky2021high} use a slightly stronger (non-central) assumption $\EE\left[\norm{\nabla f(x,\xi)}^p\right] \leq G^p$ instead of our $p$-BCM. We believe our bound is tight and can be achieved with a more careful analysis of NIGT type algorithm under the $p$-BCM assumption.

\paragraph{Second-order optimization with only Lipschitz gradient.}
Finally, we wish to discuss an important setting, which is rarely discussed in the literature on second-order optimization. Higher-order smoothness can be challenging to verify in practice. Moreover, even when it is possible, the estimates of $L_h$ can be too large to be useful. Thus we pay attention to the case when we have access to stochastic Hessian, but the second derivative can be non-continuous, i.e., $L_h \rightarrow +\infty$. Our lower bound \eqref{eq:lower_bound_second_order} provides new insights on this interesting setting, simplifying to
\begin{equation*}
    \Omega\left( \min\left\{\frac{\Delta \sigma_h}{\varepsilon^2}\left(\frac{\sigma}{\varepsilon}\right)^{\frac{1}{p-1}}, \frac{\Delta L \sigma}{\varepsilon^{3}} \left(\frac{\sigma}{\varepsilon}\right)^{\frac{1}{p-1}}\right\}\right).
\end{equation*}
As we can see, the first term corresponding to the second-order information has better dependence on an accuracy $\varepsilon$, and can be potentially much smaller than the second term. Moreover, this complexity is not impacted with the constants corresponding to higher-order smoothness, which potentially gives us the possibility of designing second-order methods with complexity that does not depend on $L_h$ and higher order smoothness constants. However, this is merely a lower bound, which does not give us any algorithmic solution yet. The question we investigate in the next section is 
\begin{center}
\textit{Can we design an algorithm handling heavy-tailed noise with complexity $\cO\rb{\frac{\Delta \sigma_h}{\varepsilon^2}\left(\frac{\sigma}{\varepsilon}\right)^{\frac{1}{p-1}}}$? }
\end{center}

\section{Near-optimal Second-Order Method under $p$-BCM}
\label{sec:NSGD_SOM_conv}

In the last decade, the {\color{black} second-order stochastic optimization (SOSO)} has been extensively studied under stronger noise assumptions. We will first review the existing approaches, particularly focusing on the development in the non-convex setup. First, \cite{tripuraneni2018stochastic} proposed and analyzed a Stochastic Cubic Newton (SCN) method, achieving $\bar\cO(\varepsilon^{-\nicefrac{7}{2}})$ sample complexity. They require a strong bounded noise assumption, use large mini-batches and importantly, this complexity is not order-optimal. Later \cite{arjevani2020second} established lower bounds for SOSO for finding first-order stationary points and proposed two algorithms: SGD with recursive variance reduction and stochastic Hessian-vector products (HVP-RVR), and subsampled cubic-regularized trust-region method with HVP-RVR. Both algorithms achieve $\bar\cO(\varepsilon^{-3})$  
sample complexities, which is minimax optimal when the variances of stochastic gradient and Hessian are bounded. Besides strong noise assumptions, their algorithms require large Hessian batch sizes, since they used HVP-RVR subroutine to update the momentum term. To overcome the large batch-size requirement, \cite{BetterSGDUsingSOM_Tran_2021} designed a simpler algorithm: SGD with Hessian-corrected momentum (with optional normalization), called SGDHess. Thanks to the Hessian corrected momentum term, their methods also achieve $\bar\cO(\varepsilon^{-3})$ sample complexity with any batch size $\geq 1$. Recently, \cite{chayti2024improving} proposed a variant of SCN with momentum without a large batch requirement. However, their analysis requires the bounded noise assumption and only achieves $\bar\cO(\varepsilon^{-\nicefrac{7}{2}})$ sample complexity.\footnote{However, in the noiseless case their complexity has better $\bar \cO(\varepsilon^{-\nicefrac{3}{2}})$ iteration complexity compared to $\bar \cO(\varepsilon^{-2})$ for SGDHess.}

To summarize, all above-mentioned works require strong noise assumptions: bounded noise and bounded variance. Moreover, their sample complexities also depend on second-order smoothness $L_h$, which can be potentially infinite, as discussed in the previous section. We refer to Table~\ref{tab:summary} for the summary of existing results. 

\begin{algorithm}[ht!]
\caption{NSGDHess (Normalized SGD with Hessian correction)}
\label{alg:NSGD_SOM}
	\begin{algorithmic}[1]
		\State \textbf{Input:} Starting point $x_0 \in \R^d$, a vector $g_0 \in \R^d$, a stepsize $\gamma > 0$, momentum parameters $\alpha >0$, the number of iterations $T$.
		\State $x_1 = x_0 - \gamma\frac{g_0}{\norm{g_0}}$
		\For{ $t =  1, 2, \ldots, T-1$}
		\State Sample $q_t \sim \cU\left([0,1]\right)$
		\State $\hat{x}_t = q_t x_t + (1-q_t)x_{t-1}$ 
		\State Sample $\xi_t, \hat{\xi}_t \sim \cD$ independently 
		\State $g_t = (1-\alpha)\left( g_{t-1} + \nabla^2 f(\hat{x}_t, \hat{\xi}_t)(x_t - x_{t-1})\right) + \alpha \nabla f(x_t, \xi_t)$
		\State $x_{t+1} = x_t -\gamma \frac{g_t}{\norm{g_t}}$
		\EndFor
		\State \textbf{Output:} 
	\end{algorithmic}
\end{algorithm}

To design an optimal second-order algorithm without these limitations, we take an inspiration from the recent developments in reinforcement learning \citep{salehkaleybar-et-al22,Fatkhullin_SPGM_FND_2023}. Their algorithms called SHARP and (N)HAR-PG are variants of SGDHess in \citep{BetterSGDUsingSOM_Tran_2021} with optional normalization step. We adapt one variant of their method to our general stochastic optimization setting and present it in Algorithm~\ref{alg:NSGD_SOM}. This algorithm computes a random interpolation point $\hat x_t$ uniformly distributed between consecutive iterates $x_{t-1}$ and $x_{t}$. The method then evaluates a stochastic gradient and a stochastic Hessian-vector product using independent samples $\xi_t$ and $\hat \xi_t$. Finally, a recursive momentum is constructed before applying the normalization step.\footnote{Normalization step for this method is important for our analysis under $p$-BCM. However, it can be removed when $p=2$, or when additional clipping is used, as in the next section.} We are now ready to state the convergence guarantee for Algorithm~\ref{alg:NSGD_SOM}.


\begin{theorem}
\label{thm:nsgdm_hess}
    Suppose Assumptions~\ref{assum:lower_bounded},~\ref{assum:l_smooth},~\ref{ass:p_BCM} and ~\ref{ass:p_BCM_hess} hold. Let the initial gradient estimate be given by   $$g_0 = \frac{1}{B_{\text{init}}}\sum^{B_{\text{init}}}_{j = 1}\nabla f(x_0, \xi_{0, j}),~~\text{where}~~B_{\text{init}} = \max\left\{1, \left(\frac{\sigma}{\varepsilon}\right)^{\frac{p}{p-1}}, \left(\frac{\sigma}{\varepsilon}\right)^{\frac{p}{2p-1}}\right\}.$$ Set 
    stepsize as $\gamma = \sqrt{\frac{\Delta \alpha^{\nicefrac{1}{p}}}{ (L + \sigma_h)  T}}$, momentum parameter as $\alpha = \min\left\{1,\alpha_{\text{eff}}\right\}$, where $\alpha_{\text{eff}} = \max\left\{\left(\frac{\cE_0}{T\sigma}\right)^{\frac{p}{2p-1}}, \left(\frac{\Delta(L+\sigma_h)}{T\sigma^2}\right)^{\frac{p}{2p-1}}\right\}$, $\cE_0 = 2\sigma/ B_{\text{init}}^{\frac{p-1}{p}}$. Then, Algorithm~\ref{alg:NSGD_SOM} guarantees that $\frac{1}{T}\sum^{T-1}_{t=0}\Exp{\norm{\nabla F(x_t)}} \leq \varepsilon$ with total sample complexity
    \begin{equation}
    \label{eq:convergence_rate_thm}
        \cO\left(\frac{\Delta (L +\sigma_h)}{\varepsilon^2} + \frac{\Delta (\sm + \sigma_h)}{\varepsilon^2} \left(\frac{\sigma}{\varepsilon}\right)^{\frac{1}{p-1}} + \frac{\sigma}{\varepsilon}\left(\frac{\sigma}{\varepsilon}\right)^{\frac{1}{p-1}} \right) .
    \end{equation}
\end{theorem}
The proof is deferred to Appendix~\ref{app:missing_proofs_NSGD_SOM}. We also investigate different choices of $g_0$ and their influence on the total sample complexity rate in Appendix~\ref{app:missing_proofs_NSGD_SOM}. We find that  while the initial batch size to estimate $g_0$ is not necessary for convergence, it  helps to slightly improve the total sample complexity.

\paragraph{Discussion:} 
Comparing this result with Theorem~\ref{thm:lower_bounds}, we observe that our upper bound \eqref{eq:convergence_rate_thm} matches our lower bound--the first term in \eqref{eq:lower_bound_second_order} in terms of the target accuracy $\varepsilon$. Moreover, when $\Delta(L+\sigma_h) \geq \sigma \varepsilon$, we have a tight upper bound, which exactly matches the lower bound from the previous section in the leading term (up to a numerical constant). Moreover, since our upper bound does not depend on constant $L_h$, it explains our observations in Section~\ref{sec:lower_bounds} regarding the limit case $L_h \rightarrow \infty, $ and answers the raised question affirmatively. We refer to Table~\ref{tab:summary} for more technical comparison with prior work.


\section{Hessian Clipping for High-probability Convergence}
\label{sec:high_probability_analysis}

In Section~\ref{sec:NSGD_SOM_conv} we provided an in-expectation guarantee for Algorithm~\ref{alg:NSGD_SOM}. While in-expectation guarantees are useful in the case when we are allowed to run a method multiple times to analyze its average-case behavior, it can be impractical. In may real-world scenarios,  only a single run of a methods can be performed due to computational or time constraints. Thus, our main goal in this section is to conduct high-probability analysis for Algorithm~\ref{alg:NSGD_SOM} or its modification. 

Since we work in the unbounded variance case, which corresponds to heavy-tailed noise, it is important to ensure that the analyzed method is robust to such noise. Following works of \cite{gorbunov2020stochastic,cutkosky2021high, sadiev2023high, nguyen2023improved}, we propose a new algorithm, called Normalized SGD with momentum and Hessian clipping. It is formally presented as  Algorithm~\ref{alg:NSGD_SOM_clipped}, where compared to Algorithm~\ref{alg:NSGD_SOM} we also incorporate gradient clipping and Hessian clipping. Formally, a clipping operator (clipping for short) is defined as 
\begin{equation}\label{eq:clipping_def}
    \texttt{clip}(v, \lambda) = \min\left\{1, \frac{\lambda}{\norm{v}}\right\} v \quad \text{for any } v \neq 0 \text{ from } \R^d, 
\end{equation}
where $\lambda >0$ is called the clipping level/threshold. 

\begin{algorithm}[ht!]
\caption{Clip NSGDHess (Normalized SGD with Hessian correction and clipping)}
\label{alg:NSGD_SOM_clipped}
	\begin{algorithmic}[1]
		\State \textbf{Input:} Starting point $x_0 \in \R^d$, a vector $g_0 \in \R^d$, a stepsize $\gamma > 0$, momentum parameters $\alpha >0$, clipping levels $\lambda > 0 , \bar \lambda_h > 0$, the number of iterations $T$.
		\State $x_1 = x_0$ and $g_0 = 0$
		\For{ $t =  1, 2, \ldots, T-1$}
		\State Sample $q_t \sim \cU\left([0,1]\right)$
		\State $\hat{x}_t = q_t \, x_t + (1-q_t) \, x_{t-1}$ 
		\State Sample $\xi_t, \hat{\xi}_t \sim \cD$ independently 
		\State $g_t = (1-\alpha)\left( g_{t-1} + \gamma \, \texttt{clip}\left(\gamma^{-1} \nabla^2 f(\hat{x}_t, \hat{\xi}_t)(x_t - x_{t-1}), \bar \lambda_h \right)\right) + \alpha \, \texttt{clip}\left(\nabla f(x_t, \xi_t), \lambda\right) $
		\State $x_{t+1} = x_t -\gamma \frac{g_t}{\norm{g_t}}$
		\EndFor
		\State \textbf{Output:} 
	\end{algorithmic}
\end{algorithm}

As in Algorithm~\ref{alg:NSGD_SOM}, we use normalization in the gradient update (line 8), which allows the next point $x_{t+1}$ to stay in the ball with center at $x_t$ of radius $\gamma$. The main difference between Algorithms~\ref{alg:NSGD_SOM} and~\ref{alg:NSGD_SOM_clipped} is in the momentum term. To enhance robustness, we do not only clip the gradient, but also the Hessian-vector product. As we mentioned above, gradient clipping is a common approach, but Hessian clipping is new, and it allows us to provide the first high-probability guarantees for second-order stochastic optimization without light-tail noise assumptions. It is worth to draw attention to the term
\begin{equation}\label{eq:hessian_clipping_term}
\gamma \, \texttt{clip}\left(\gamma^{-1} \nabla^2 f(\hat{x}_t, \hat{\xi}_t)(x_t - x_{t-1}), \bar \lambda_h \right).
\end{equation}
Observe that clipping the entire Hessian based on its operator norm is very costly since it typically requires $\cO(d^2)$ arithmetic operations (e.g., using power iteration \citep{golub2013matrix} or Lanczos algorithm \citep{lanczos1950iteration,saad2011numerical}). 
Instead of clipping the stochastic Hessian, we clip the Hessian-vector product, which is computationally tractable, and only requires $\cO(d)$ operations. Thus, the computation of this clipped Hessian-vector product can be easily implemented via backpropagation and subsequent ``vector'' clipping \eqref{eq:clipping_def}. Another reason why we prefer to use this form of clipping comes from the analysis; we want to preserve the following useful property: 
\begin{eqnarray}\label{eq:expectation_of_hv}
    \EE_{q_t, \hat{\xi}_t}\left[\nabla^2 f(\hat{x}_t, \hat{\xi}_t)(x_t - x_{t-1})\right] &=& \EE_{q_t}\left[\EE_{ \hat{\xi}_t}\left[\nabla^2 f(\hat{x}_t, \hat{\xi}_t)\right](x_t - x_{t-1})\right] \notag \\
    &=& \EE_{q_t}\left[\nabla^2 F(\hat{x}_t)(x_t - x_{t-1})\right] \notag \\
    &=&\int^1_0 \nabla^2 F(q x_t + (1-q)x_{t-1})(x_t - x_{t-1})dq \notag \\
    &=& \nabla F(x_t) - \nabla F(x_{t-1}). 
\end{eqnarray}
This would not hold were we to use direct Hessian matrix clipping. Instead, we prove that the expectation of \eqref{eq:hessian_clipping_term} will approximately be equal to \eqref{eq:expectation_of_hv} when $\bar \lambda_h$ is selected properly. By smoothness and gradient step, we have $\|\nabla F(x_t) - \nabla F(x_{t-1})\| \leq L\|x_t-x_{t-1}\|= L\gamma$, meaning the expectation is bounded. 
To provide the analysis, we need to rewrite the clipped Hessian-vector product as follows 
$$\gamma \, \texttt{clip}\left(\gamma^{-1} \nabla^2 f(\hat{x}_t, \hat{\xi}_t)(x_t - x_{t-1}), \bar \lambda_h \right) = \texttt{clip}\left( \nabla^2 f(\hat{x}_t, \hat{\xi}_t)(x_t - x_{t-1}), \gamma \bar \lambda_h \right),$$
where we denote $\lambda_h = \gamma \bar \lambda_h$. The former formulation in the above formula is useful for implementation, since as we will see, $\lambda_{h} \sim \cO(\gamma\alpha^{-\nicefrac{1}{p}})$ and $\lambda \sim \cO(\alpha^{-\nicefrac{1}{p}})$. Thus, $\lambda$ and $\bar{\lambda}_h$ are of the same order $\sim\cO(\alpha^{-\nicefrac{1}{p}}) = \cO(T^{\frac{1}{2p-1}})$, which makes it easier to tune these two parameters together in practice. The latter formulation will be used in the high-probability analysis. We are now ready to state the main theorem of this section.

\begin{theorem}
\label{thm:main_theorem_hp}
Let Assumptions~\ref{assum:lower_bounded},\ref{assum:l_smooth},\ref{ass:p_BCM}, and~\ref{ass:p_BCM_hess} hold, and define the initial optimality gap as $\Delta_1 := F(x_0) - F_*$. Let $\delta \in (0,1]$ and $T \geq 1$ be such that $\log(8T/\delta) \geq 1$.
Suppose Algorithm~\ref{alg:NSGD_SOM_clipped} is executed with the following settings:  momentum parameter $\alpha = T^{-\frac{p}{2p - 1}}$,  clipping levels
        $\lambda = \max\left\{4\sqrt{L\Delta_1}, \; \frac{\sigma}{\alpha^{1/p}} \right\},
    ~~
        \lambda_h = \frac{2\gamma(L + \sigma_h)}{\alpha^{1/p}},$
 stepsize $\gamma$ satisfying:
    \begin{eqnarray*}
        \gamma \leq \cO\left(\min \Bigg\{
        \sqrt{\frac{\Delta_1}{LT}}, \;
       \alpha \sqrt{\frac{\Delta_1}{L}}, \;
        \frac{\sqrt{\nicefrac{\Delta_1}{L}}}{\alpha T \log(T/\delta)} , 
        \frac{\Delta_1}{ \sigma \alpha^{\frac{p-1}{p}} T \log(T/\delta)}, \;
        \sqrt{ \frac{\Delta_1 \alpha^{1/p}}{(L + \sigma_h)T \log(T/\delta)} }
        \Bigg\}\right).
    \end{eqnarray*}
Then, with probability at least $1 - \delta$, the output of the algorithm satisfies:
$\frac{1}{T} \sum_{t=0}^{T-1} \|\nabla F(x_t)\| \leq \frac{2\Delta_1}{\gamma T}.$
As a result, the gradient norm converges at the rate
\begin{equation*}
    \frac{1}{T} \sum_{t=0}^{T-1} \|\nabla F(x_t)\|
        = \mathcal{O} \left(
        \frac{ \sqrt{\Delta_1(L + \sigma_h)}+ \sigma }{T^{\frac{p - 1}{2p - 1}}}
        \log \left( \frac{T}{\delta} \right)
        \right)
\end{equation*}
with high probability.
\end{theorem}
The complete proof of Theorem~\ref{thm:main_theorem_hp} can be found in Appendix~\ref{app:missing_proofs_hp}. Disregarding the parameters $L, \sigma, \sigma_h, \Delta_1 $ in the choices of  the stepsize  and the clipping level, we see that the momentum parameter and the stepsize decrease with the same rate $\sim \cO(\nicefrac{1}{T^{\frac{p}{p-1}}})$, similarly to Algorithm~\ref{alg:NSGD_SOM}. Also, as we mentioned above, clipping levels $\lambda$ and $\bar \lambda_h$ have the same rate, and behave as $\sim \cO({T^{\frac{1}{2p-1}}})$. 

\begin{corollary}
    \label{cor:high_probability}
    In the setting of Theorem~\ref{thm:main_theorem_hp}, Algorithm~\ref{alg:NSGD_SOM_clipped} ensures that $\frac{1}{T}\sum^{T-1}_{t=0}\|\nabla F(x_t)\| \leq \varepsilon$, with probability at least $1-\delta$. To achieve this, the algorithm requires at most
    \begin{equation*}
         \cO\left(\left(\frac{\sqrt{\Delta_1(L+\sigma_h)}+ \sigma}{\varepsilon}\log \frac{1}{\delta\varepsilon}\right)^{\frac{2p-1}{p-1}}\right)\quad \text{first/second-order oracle calls.}
    \end{equation*}  
\end{corollary}

According to the results of Theorem~\ref{thm:main_theorem_hp} and Corollary~\ref{cor:high_probability},  we have shown that Algorithm~\ref{alg:NSGD_SOM_clipped} has a near-optimal convergence rate in the dependence on $\varepsilon$. Moreover, the complexity has a logarithmic dependence on $\nicefrac{1}{\delta}$, which is better than a polynomial dependence that can be achieved by directly translating the in-expectation result of Theorem~\ref{thm:nsgdm_hess} to a high probability one using Markov's inequality. Unfortunately, the dependence on $\Delta_1, L, \sigma, \sigma_h$ does not match the  lower bound \eqref{eq:lower_bound_second_order}. 

\paragraph{Comparison with \citep{UnboundedClippedNSGDM2023Liu}.} Recently, and under additional strong assumptions, recently \cite{UnboundedClippedNSGDM2023Liu} analyzed a momentum-based variance-reduced first-order method, and obtained rates similar those in our Corollary~\ref{cor:high_probability}. However, they assume individual smoothness (i.e., $\|\nabla f(x,\xi) - \nabla f(y, \xi)\| \leq \ell \|x - y\|$ for all $x, y\in \R^d$ and all realizations of the random variable $\xi$). However, individual smoothness is a very strong assumption which implies boundedness of the stochastic Hessian, i.e., $\|\nabla^2 f(x, \xi)\| \leq \ell$ a.s., for all $x\in \R^d$. We also note that under individual smoothness,  fast $p$-independent rates can be obtained, e.g., \citet{lei2019stochastic} achieve $\bar \cO(\varepsilon^{-4})$ complexity for SGD under  individual smoothness.

\section{Limitations and Future Work}
While our proposed method with Hessian clipping has many desirable properties, such as near-optimal sample complexity, batch-free and high-probability convergence, under mild statistical assumptions, it also has several limitations. First, it requires tuning three extra parameters compared to, e.g., vanilla SGD \citep{fatkhullin2025can} or Normalized SGD \citep{Hubler2024clip_to_norm}. It would be interesting to investigate if such extensive tuning can be removed for second-order methods. Second, while our sample complexity is near optimal, it does not match our lower bound exactly in all important parameters. Moreover, while we mainly pay attention to the leading stochastic terms in the complexity, our methods do not recover optimal deterministic complexities of second-order methods. Perhaps a different analysis technique for clipping and some algorithmic modifications (like Newton type preconditioning \citep{chayti2024improving}) are required to exactly match the upper and lower bounds. Third, high probability upper bounds are derived for a fixed confidence level $\delta$ rather than uniformly for all $\delta \in (0,1)$ (as e.g., in \citep{liu2023high,Hubler2024clip_to_norm}) and the algorithm's parameters depend on this choice. This is limiting as it does not necessarily imply high probability convergence for a fixed parameter choice. Finally, while we make an important progress in the fundamental understanding of second-order stochastic optimization, such results should always be interpreted with caution. In practice, the overhead of second-order methods (computing Hessian information or even Hessian-vector products) and the challenge of tuning additional hyperparameters (such as clipping thresholds for both gradients and Hessians) can be significant. 

\section*{Acknowledgements}
The research reported in this publication was supported by funding from King Abdullah University of Science and Technology (KAUST): i) KAUST Baseline Research Scheme, ii) CRG Grant ORFS-CRG12-2024-6460, and iii) Center of Excellence for Generative AI, under award number 5940. The work of I.F. is supported by ETH AI Center Doctoral Fellowship. 

\bibliography{main}

\newpage
\appendix
\tableofcontents
\newpage
\section{Additional Related Work}
\paragraph{Gradient clipping and normalization.} Gradient clipping has been applied successfully also in zero-order optimization \citep{kornilov2024accelerated}, bandit and RL literature \citep{bubeck2013bandits,cayci2024provably}, online learning \citep{zhang2022parameter} and differential privacy \citep{abadi2016deep,sha2024clip}. Some other works, which use normalization and gradient clipping under heavy-tailed noise include \citep{NonlinearHP2023Armacki,puchkin2024breaking,liu2024nonconvex}.

\begin{table}[t]
    \centering
    \caption{Summary of sample complexities of stochastic second-order (SOSO) methods for finding an $\varepsilon$-stationary point, in high probability of in expectation, i.e., number of stochastic gradient and Hessian evaluations to find $\bar x$ with $\norm{\nabla F(\bar x)} \leq \varepsilon.$ The column ``$p$'' indicates the range of moments in Assumptions~\ref{ass:p_BCM}, \ref{ass:p_BCM_hess} for which the result holds, that is when $p=2$, the corresponding result holds only under bounded variance. When $p=\infty$ it means that at least for stochastic gradient or stochastic Hessian bounded noise assumption is required, which corresponds to all moments being bounded. We are not aware of any prior works for SOSO for $p<2.$ The column ``\textbf{HP?}'' denotes whether the high probability guarantee with polylogarithmic dependence on the inverse of failure probability $1/\delta$ is available. }
    \label{tab:summary}
 \begin{threeparttable}
\resizebox{\textwidth}{!}{%
\begin{tabular}{cccc}

\bf Algorithm & \bf Sample Complexity &  $p$  & \bf HP?\\ \hline \hline 
\begin{tabular}{c} SCN \\ \citep{tripuraneni2018stochastic} \end{tabular} & $\frac{\Delta\sqrt{L_h}\sigma^2}{\varepsilon^{\nicefrac{7}{2}}} $ & {\color{BrickRed}$\infty$} & \cmark$^{\color{blue}(1)}$ \\ 
\begin{tabular}{c} SGD with HVP-RVR \\ \citep{arjevani2020second} \end{tabular} & $\frac{\Delta\sigma\sigma_h}{\varepsilon^{3}} + \frac{\Delta \sqrt{L_h}\sigma_1}{\varepsilon^{\nicefrac{5}{2}}}+ \frac{\Delta L}{\varepsilon^2}$ & {\color{BrickRed}$2$} & \xmark\\ 
\begin{tabular}{c} SN$^{\color{blue}(2)}$ with  HVP-RVR  \\ \citep{arjevani2020second} \end{tabular} &  $\frac{\Delta\sigma\sigma_h}{\varepsilon^{3}} + \frac{\Delta \sqrt{L_h}\sigma_1}{\varepsilon^{\nicefrac{5}{2}}}+ \frac{\Delta \sigma_h}{\varepsilon^2}$  & {\color{BrickRed}$2$} & \xmark\\
\begin{tabular}{c} N-SGDHess   \\ \citep{BetterSGDUsingSOM_Tran_2021} \end{tabular} &  $\frac{\sigma^3}{\varepsilon^{3}}+ \frac{\Delta\sigma\sigma_h}{\varepsilon^{3}} + \frac{\Delta \sqrt{L_h}\sigma_1}{\varepsilon^{\nicefrac{5}{2}}}+ \frac{\Delta \sigma_h}{\varepsilon^2}$  & {\color{BrickRed}$2$}& \xmark \\
\begin{tabular}{c} SCN with  IT-HB  \\ \citep{chayti2024improving} \end{tabular} &  $\frac{\Delta\sqrt{L_h}}{\varepsilon^{\nicefrac{3}{2}}} + \frac{\Delta L_h^{\nicefrac{1}{4}}\sigma_h^{\nicefrac{1}{2}}}{\varepsilon^{\nicefrac{7}{4}}}+ \frac{\Delta \sqrt{L}\sigma^2}{\varepsilon^{\nicefrac{7}{2}}}$  & {\color{BrickRed}$\infty$}& \xmark\\
\hline\hline
 \cellcolor{linen} \begin{tabular}{c} Lower Bounds \\   Theorem~\ref{thm:lower_bounds} \end{tabular}
& \cellcolor{linen}  $\min\left\{\frac{\Delta \sigma_h}{\varepsilon^2}, \frac{\Delta L \sigma}{\varepsilon^{3}} , \frac{\Delta \sqrt{L_h}\sigma}{\varepsilon^{\nicefrac{5}{2}}} \right\} \left(\frac{\sigma}{\varepsilon}\right)^{\frac{1}{p-1}}$ &\cellcolor{linen} {\color{PineGreen}$(1,2]$} & \cellcolor{linen}\\
\hline
 \cellcolor{linen} \begin{tabular}{c} NSGDMHess \\   Theorem~\ref{thm:nsgdm_hess} \end{tabular}
& \cellcolor{linen} $\frac{\Delta (L +\sigma_h)}{\varepsilon^2} + \left(\frac{\Delta (\sm + \sigma_h)}{\varepsilon^2}  + \frac{\sigma}{\varepsilon}\right)\left(\frac{\sigma}{\varepsilon}\right)^{\frac{1}{p-1}} $  &\cellcolor{linen} {\color{PineGreen}$(1,2]$} & \cellcolor{linen} \xmark \\
 \cellcolor{linen} \begin{tabular}{c} Clip NSGDMHess  \\  Theorem~\ref{thm:main_theorem_hp} \end{tabular} & \cellcolor{linen}$\left(\frac{\sqrt{\Delta_1(L+\sigma_h)}+ \sigma}{\varepsilon}\right)^{2 +\frac{1}{p-1}}$ & \cellcolor{linen} {\color{PineGreen}$(1,2]$} & \cellcolor{linen} \cmark \\ 
\end{tabular}
}
\begin{tablenotes}
        {\scriptsize \item [{\color{blue}(1)}] 
        \cite{tripuraneni2018stochastic} provide analysis under stronger  assumptions, which implies the noise has light tails.  
        \item [{\color{blue}(2)}] SN = Subsampled Newton.
        \item [{\color{blue}(3)}]
        }
    \end{tablenotes}
    \end{threeparttable}
\end{table}

\section{About Normalized SGD and Different Variants of Momentum}
In this section we describe a general framework of momentum terms. There is a lot of literature about it especially for convex optimization, since \cite{nesterov1} showed that an appropriate use of momentum with GD accelerates convergence. But we want to emphasize that in our work we consider momentum in the context of stochastic non-convex optimization. The general update rules is defined as follows 
\begin{equation*}
    x_{t+1} = x_t - \gamma_t\frac{g_t}{\norm{g_t}},
\end{equation*}
where $g_t$ is a momentum term. Different selections of $g_t$ leads to different types of methods. Let us consider several of them.

\textbf{NSGD} is the well-known method. if $g_t = \nabla f(x_t,\xi_t)$, we immediately get NSGD. As it was shown that NSGD converges only to the neighborhood, which radius is equal to $\sigma$ \citep{yang2023two_sides,Hubler2024clip_to_norm}. One of the possible way to fix it is mini-batching, allowing to reduce the variance. But if we want to use only one sample per iteration, this strategy cannot be applied. 

\textbf{NSGDM} is the solution of the above problem and was proposed and analyzed by \cite{MomentumImprovesNSGD_Cutkosky_2020}. The momentum term is defined as 
\begin{equation*}
    g_{t} = (1-\alpha)g_{t-1} + \alpha \nabla f(x_t, \xi_t).
\end{equation*}
In their work they showed that NSGDM converges to $\varepsilon$-stationary point and  has sample complexity, which is equal to $\cO(\varepsilon^{-4})$, which is optimal according to \cite{arjevani2023lower}. Later, \citet{Hubler2024clip_to_norm} extended this result to the case of bounded $p$-th moment deriving $\cO(\varepsilon^{-\frac{3p-2}{p-1}})$ sample complexity.

\textbf{NSGDM with clipping} was proposed directly for the general set up with bounded $p$-th moment of stochastic gradient. The momentum term is modified as follows 
\begin{equation*}
    g_{t} = (1-\alpha)g_{t-1} + \alpha \texttt{clip}\left(\nabla f(x_t, \xi_t), \lambda\right).
\end{equation*}
\cite{cutkosky2021high} provided high-probability convergence guarantees, and the sample complexity of this method is $\cO(\varepsilon^{-\frac{3p-2}{p-1}})$, which is optimal and matches with lower bounds from the paper of \cite{zhang2020adaptive}. It is worth to mention that the clipping operator is not necessary and normalization of momentum term is sufficient to guarantee convergence under Assumption~\ref{ass:p_BCM} \citep{Hubler2024clip_to_norm,liu2024nonconvex}. 

\textbf{NIGT} is NSGD with momentum and implicit gradient transportation proposed by \cite{MomentumImprovesNSGD_Cutkosky_2020}. They modified momentum term with replacing  the point, at which stochastic gradient is computed:
\begin{equation*}
    g_{t} = (1-\alpha)g_{t-1} + \alpha \nabla f(y_t, \xi_t),\quad \text{where}~y_t = x_t +\frac{1-\alpha}{\alpha}x_{t-1}.
\end{equation*}
Thanks to this modification, assuming that the objective function has Lipschitz continuous Hessian \cite{MomentumImprovesNSGD_Cutkosky_2020} proved NIGT has sample complexity $\cO\left(\varepsilon^{-\nicefrac{7}{2}}\right)$, which is better than NSGDM. 

\textbf{NIGT with clipping} was motivated for high-probability analysis under bonded $p$-th moment of stochastic gradient by \cite{cutkosky2021high}. The main difference from NIGT is only a clipping operator, which is needed from theoretical perspective. According to \cite{cutkosky2021high}, NIGT with clipping has sample complexity $\cO\left(\varepsilon^{-\frac{5p-3}{2p-2}}\right)$, which is optimal according to Theorem~\ref{thm:lower_bounds}.

\textbf{NSGD with MVR} is other version of NSGD with momentum variance reduction (MVR). STORM/MVR was proposed by \cite{cutkosky-orabona19} to improve the sample complexity compared to SGD.

The above considered methods use only first-order information. What if we try to exploit second-order information to improve sample complexity? 

\textbf{NSGD-Hess} is Normalized SGD with Hessian-corrected momentum proposed by \cite{BetterSGDUsingSOM_Tran_2021}. They introduced stochastic Hessian-vector product into momentum term:
\begin{equation*}
    g_{t} = (1-\alpha)\left(g_{t-1}+\nabla^2 f(x_t,\xi_t)\cdot (x_t-x_{t-1}) \right)+ \alpha \nabla f(y_t, \xi_t),
\end{equation*}
\textbf{Refined NSGD-Hess.} Our refined momentum modifies the point where the stochastic Hessian is evaluated and uses another fresh sample $\hat \xi_t \sim \cD$ for the Hessian-vector product computation. 
\begin{eqnarray*}
    q_t &\sim& \cU\left([0,1]\right), \\
    \hat{x}_t &=& q_t x_t + (1-q_t)x_{t-1}, \\
    g_{t} &=& (1-\alpha)\left(g_{t-1}+\nabla^2 f(\hat x_t,\hat \xi_t)\cdot (x_t-x_{t-1}) \right)+ \alpha \nabla f(x_t, \xi_t) . 
\end{eqnarray*}

\newpage 
\section{Technical Lemmas}
We list the following technical lemmas, which are useful for our analysis in the subsequent sections.
\begin{lemma}[Lemma 10 from \cite{Hubler2024clip_to_norm}]
\label{lem:von_Bahr_inequality}
    Let $p \in [1;2]$, and $X_1,\dots, X_n \in \R^d$  be a martingale difference sequence, i.e. $\EE\left[X_j| X_{j-1},\dots, X_1\right] = 0$ a.s. for all $j=1,\dots, n$ satisfying 
    \begin{equation*}
        \EE\left[\norm{X_j}^p\right] < \infty \quad \quad \text{for all } i=1,\dots, n.
    \end{equation*}
    Define $S_n \eqdef \sum^n_{j=1}X_j$, then 
    \begin{equation*}
        \EE\left[\norm{S_n}^p\right] \leq 2\sum^n_{i=j}\EE\left[\norm{X_j}^p\right].
    \end{equation*}
\end{lemma}

\begin{lemma}[Lemma 10 from \cite{cutkosky2021high}]
\label{lem:inequlity_from_cutkosky}
    Let  $X_1,\dots, X_n \in \R^d$  be a sequence of random vectors. Define the sequence of real numbers $w_1,\dots, w_n$ recursively 
    \begin{enumerate}
        \item $w_0$ = 0
        \item If $\sum^{j-1}_{i=1}X_i \neq 0$, then we set: 
        \begin{equation*}
            w_j = \sign\left(\sum^{j-1}_{i=1}w_i\right)\frac{\left\la \sum^{j-1}_{i=1}X_i, X_j\right\ra}{\norm{\sum^{j-1}_{i=1}X_i}}.
        \end{equation*}
        \item If $\sum^{j-1}_{i=1}X_i = 0$, set $w_j = 0$.
    \end{enumerate}
    Then  $|w_j| \leq \norm{X_j}$ for all $j=1,\dots, n$, and 
    \begin{equation}
        \norm{S_n}\leq \left|\sum^n_{j=1} w_j\right| + \sqrt{\max_{j \in [n]}\norm{X_j}^2 + \sum^n_{j=1}\norm{X_j}^2}.
    \end{equation}
\end{lemma}

\begin{lemma}[Bernstein inequality]
\label{lem:Bernstain_inequality}
    Let  $X_1,\dots, X_n \in \R^d$  be a martingale difference sequence, i.e. $\EE\left[X_j| X_{j-1},\dots, X_1\right] = 0$ a.s. for all $j=1,\dots, n$. Assume that conditional variances $\sigma^2_{j} \eqdef \EE\left[X^2_j| X_{j-1},\dots, X_1\right]$ exist and are bounded, and assume that there exists deterministic constant $c > 0$ such that $|X_j| \leq c$ almost surely for all  $j=1,\dots, n$. Define $S_n \eqdef \sum^n_{i=1}X_j$, then for all $b > 0, G> 0$ 
    \begin{equation}
        \PP\left\{\left|S_n\right| > b \text{ and } \sum^n_{j=1}\sigma^2_i \leq G\right\} \leq 2\exp\left( - \frac{b^2}{2G + \nicefrac{2cb}{3}}\right),
    \end{equation}
\end{lemma}

\begin{lemma}[ Lemma 5.1 from \cite{sadiev2023high}]
\label{lem:bound_variance_magnitude}
    Let $\lambda > 0$ and $X \in \R^d$ be a random vector and $\widetilde{X} \eqdef \texttt{clip}\left(X, \lambda\right)$. Then, $\norm{\widetilde{X} - \Exp{\widetilde{X}}} \leq 2\lambda$. Moreover, if for some $\sigma \geq 0$ and $p \in (1;2]$ we have $\Exp{X} = x \in \R^d$, $\Exp{\norm{X-x}^p}\leq \sigma^p$, and $\norm{x} \leq \nicefrac{\lambda}{2}$, then 
    \begin{equation*}
        \norm{\Exp{\widetilde{X}} -x} \leq \frac{2^p\sigma^p}{\lambda^{p-1}},\quad \Exp{\norm{\widetilde{X} - x}^2}\leq 18\lambda^{2-p}\sigma^p,\quad \Exp{\norm{\widetilde{X} - \Exp{\widetilde{X}}}^2}\leq 18\lambda^{2-p}\sigma^p.
    \end{equation*}
\end{lemma}

\newpage
\section{Missing Proofs for Section~\ref{sec:lower_bounds}}
\label{app:missing_proofs_lower_bounds}
In this section we provide lower bounds for the certain class of functions and the oracle class. Our prove is inspired by the paper of \cite{arjevani2020second}.

\textbf{Auxiliary Lemmas.} Now we state two helpful lemmas to prove lower bound. The first lemma is about the number of iteration to reveal all coordinates, when zero-respecting algorithms use the information received from the oracle forming \textit{probability-$\rho$ zero-chain}. 
\begin{definition}
    A collection of derivative estimators ${\nabla^1 f}(x,\xi)$, ${\nabla^2 f}(x,\xi)$, $\dots$, ${\nabla^q f}(x,\xi)$ for a function $F$ forms a probability-$\rho$ zero-chain if 
    \begin{equation*}
        \PP\left\{\exists~ x ~|~ \text{prog}\left({\nabla^1 f}(x,\xi), \dots, {\nabla^q f}(x,\xi)\right) = \text{prog}_{\frac{1}{4}}(x) +1 \right\} \leq \rho 
    \end{equation*}
    and 
    \begin{equation*}
        \PP\left\{\exists~ x ~|~ \text{prog}\left({\nabla^1 f}(x,\xi), \dots, {\nabla^q f}(x,\xi)\right) = \text{prog}_{\frac{1}{4}}(x) + i \right\} = 0,~ i > 1.
    \end{equation*}
\end{definition}
The second lemma is about the main properties of the \textit{worst} function from  the class $\cF(\Delta, L_{1:q})$. 

\begin{lemma}[\citep{arjevani2020second}]
\label{lem:lower_bound_1}
    Let ${\nabla^1} f(x,\xi), \dots, {\nabla^q} f(x,\xi)$ be a collection of probability-$\rho$ zero-chain derivative estimators for $F: \R^{T}\rightarrow \R$, and let $\text{O}^q_F$ be an oracle with $\text{O}^p_f = \left({\nabla^r f}(x,\xi)\right)_{r \in [q]}$. Let $\{x^{(t)}_{A[\text{O}_F^q]}\}$ be a sequence of queries produced by a zero-respecting algorithm $A$ interacting with $\text{O}^q_F$. Then, with probability at least $1-\delta$
    \begin{equation*}
        \text{prog}\left(x^{(t)}_{A\left[O^q_F\right]}\right) < T, \qquad \text{ for all } \quad t \leq \frac{T-\log\frac{1}{\delta}}{2\rho} .
    \end{equation*}
\end{lemma}

\begin{lemma}[\cite{carmon2020lower}]
\label{lem:lower_bound_2}
Let $h: \R^T\rightarrow \R$ be the following function:
\begin{equation*}
    h(x) = -\Psi(1)\Phi(x_1) + \sum^T_{i=2}\left(\Psi(-x_{i-1})\Phi(-x_i) - \Psi(x_{i-1})\Phi(x_i)\right),
\end{equation*}
where 
\begin{equation*}
    \Psi(x) = \begin{cases}
        0,& x \leq \frac{1}{2};\\
        \exp\left(1-\frac{1}{(2x-1)^2}\right),& x > \frac{1}{2},
    \end{cases}
    \quad\text{and}\quad \Phi(x) = \sqrt{e}\int\limits^x_{-\infty} e^{-\frac{1}{2}t^2} dt. 
\end{equation*}
Then the function $h$ satisfies the following properties: 
    \begin{itemize}
        \item $h(0) - \inf\limits_x h(x) \leq \Delta_0 T$, where $\Delta_0 = 12$. 
        \item For $q \geq 1$, the $q$-th order derivatives of $h$ are $\ell_q$-Lipschitz continuous, where $\ell_q \leq \exp\left(\frac{5}{2}q\log q + cq\right)$ for a numerical constant $c < \infty$. 
        \item For all $x \in \R^{T}$, $q \in \mathbb{N}
        $ and $i \in [T]$, we have $\|\nabla^q_i h(x)\|_{\text{op}} \leq \ell_{q-1}$, where $\ell_0 = 23$. 
        \item For all $x \in \R^T$ and $q \in \mathbb{N} $, $\text{prog}\left(\nabla^q h(x)\right) \leq \text{prog}_{\frac12} (x) +1 $.
        \item For all $x \in \R^T$, if $\text{prog}_1(x) < T$, then $\|\nabla h(x)\| \geq |\nabla_{\text{prog}_1(x)+1} h(x)| > 1$.
    \end{itemize}
\end{lemma}

\begin{lemma}
\label{lem:lower_bound_3}
    Let the derivative estimator is defined as follows for each $r \in [q]$
    \begin{equation*}
        \left[{\nabla^r h}(x,\xi)\right]_i \eqdef \left(1 + \mathbb{I}\left\{ i > \text{prog}_{\frac{1}{4}}(x)\right\}\left(\frac{\xi}{\rho}-1\right)\right) \cdot \nabla^r_i h(x),
    \end{equation*}
    where $\xi \sim \text{Bernoulli}(\rho)$. Then, $\{\nabla^r h(x,\xi)\}_{r \in [q]}$ forms probability-$\rho$ zero-chain, and  for any $p \in [1,2]$ and for each $r \in [q]$ the derivative estimator satisfies 
    \begin{equation*}
        \EE\left[{\nabla^r h}(x,z)\right] = {\nabla^r h}(x),\quad\text{and}\quad \EE\left[\left\|{\nabla^r h}(x,\xi) - {\nabla^r h}(x)\right\|^p_{\text{op}}\right] \leq \frac{2\ell^p_{r-1}}{\rho^{p-1}}  \qquad \text{for all } x\in\R^d.
    \end{equation*}
\end{lemma}
\begin{proof}
    First of all, it is easy to show that the proposed estimator is unbiased, i.e.
    \begin{equation*}
        \EE\left[{\nabla^r h}(x,\xi)\right] = (1+0)\cdot \nabla^r h(x) = \nabla^r h(x).
    \end{equation*}
    According to Lemma~\ref{lem:lower_bound_2}, we have  $\text{prog}\left(\nabla^r h(x)\right) < \text{prog}_{\frac12}(x) +1 \leq \text{prog}_{\frac14}(x) +1$, where the last inequality is true because $\text{prog}_{\alpha}$ is non-increasing with respect to $\alpha$. Then, by Lemma~\ref{lem:lower_bound_2} we have 
    \begin{equation*}
        \left[{\nabla^r h}(x,\xi)\right]_i = \nabla^r_i h(x) = 0,\quad \forall i > \text{prog}_{\frac14}(x) + 1.
    \end{equation*}
    Therefore, due to that $\xi \sim \text{Bernoulli}(\rho)$  we obtain
    \begin{equation*}
        \PP\left\{\exists x~|~ \text{prog}\left({\nabla^1 f}(x,z), \dots, {\nabla^q f}(x,z)\right) = \text{prog}_{\frac{1}{4}}(x) +1 \right\} \leq \rho,
    \end{equation*}
    i.e. $\{\nabla^r h(x,\xi)\}_{r \in [q]}$ forms probability-$\rho$ zero-chain.

Now we bound $p$-th moment of ${\nabla^r h}(x,\xi) -\nabla^r h(x)$ for any $p \in [1;2]$. Denoting $i_x = \text{prog}_{\frac14}(x) +1 $, for any $r \in [q]$  we have 
\begin{eqnarray*}
    \EE\left[\left\|{\nabla^r h}(x,\xi) -\nabla^r h(x) \right\|^{p}_{\text{op}}\right] &=& \EE\left[\left|\mathbb{I}\left\{ i_x > \text{prog}_{\frac14}(x)\right\}\left(\frac{\xi}{\rho}-1\right)\right|^{p} \|\nabla^r_i h(x)\|^{p}_{\text{op}}\right]\\
    &=& \EE\left[\left|\frac{\xi}{\rho}-1\right|^{p}\right] \|\nabla^r_i h(x)\|^{p}_{\text{op}}\\
    &\leq& \ell^{p}_{r-1}\EE\left[\left|\frac{\xi}{\rho}-1\right|^{p}\right]\\
    &=&  \ell^{p}_{r-1} \frac{(1-\rho)^{p}\rho + (1-\rho)\rho^{p}}{\rho^{p}}\\
    &\leq&   \ell^{p}_{r-1} \frac{2}{\rho^{p-1}}. 
\end{eqnarray*}
\end{proof}

\textbf{Main Theorem.} Now we are ready to state and prove the main theorem in this section, i.e. Theorem~\ref{thm:lower_bounds}. 

\begin{proof} First of all, we fix all parameters: $q\in \mathbb{N}, \Delta, L_{1:q}, \sigma_{1:q} > 0 $  and $\varepsilon > 0$. Next, we rescale our function $h$ as follows: $h^{\star}(x) = \nu h(\beta x)$, where $\nu> 0$ and $\beta >0$.
According to Lemma~\ref{lem:lower_bound_2}, selecting $T = \left\lfloor \frac{\Delta}{\nu\Delta_0}\right\rfloor$, we have  for any $r \in [q-1]$ 
\begin{eqnarray*}
    h^{\star}(0) - \inf\limits_x h^{\star}(x) &=& \nu\left(h^{\star}(0) - \inf\limits_x h^{\star}( \beta x)\right) \leq \nu \Delta_0 T \leq \Delta;\\
    \|\nabla^{r+1} h^{\star}(x)\|_{\text{op}} &=& \nu \beta^{r+1}\|\nabla^{r+1} h(\beta x)\| \leq \nu \beta^{r+1}\ell_r;\\
    \|\nabla h^{\star}(x)\| &=& \nu\beta\|\nabla h(\beta x)\| \geq \nu \beta \|\nabla h(x)\| \geq \nu \beta, ~\forall~x:~\text{prog}_1(x) < T. 
\end{eqnarray*}
Since $\left\{{\nabla^r h^{\star}}(x,\xi) = \nu\beta^r{\nabla^r h}(\beta x,\xi) \right\}_{r\in [p]}$ forms a probability-$\rho$ zero-chain, then by Lemma~\ref{lem:lower_bound_2}  with probability at least $\delta = \frac{1}{2}$ we have 
\begin{equation*}
    \EE\left[\left\|\nabla h^{\star}(x^{t}_{A[O_h^q]})\right\|\right] = \nu \beta \EE\left[\left\|\nabla  h^{\star}(\beta x^{t}_{A[O_h^q]})\right\|\right] \geq \frac{\nu \beta}{2},\quad \forall~t\leq \frac{T-1}{2\rho}
\end{equation*}
The $p$-th  central moment of the scaled derivatives estimators is bounded as 
\begin{eqnarray*}
    \EE\left[\left\|{\nabla^r  h^{\star}}(x,\xi) - \nabla^r  h^{\star} (x)\right\|_{\text{op}}^{p}\right] &\leq& \nu^{p}\beta^{rp} \EE\left[\left\|{\nabla^r h}(\beta x,\xi) - \nabla^r h (\beta x)\right\|_{\text{op}}^{p}\right]\\
    &\leq& \frac{2\nu^{p}\beta^{rp}\ell_{r-1}^{p}}{\rho^{p-1}},
\end{eqnarray*}
where in the last inequality we applied Lemma~\ref{lem:lower_bound_3}.
Thus, we have the set of constraints for parameters $\nu$ and $\beta$ for all $r \in [q]$:
\begin{equation*}
    \begin{cases}
        \nu \Delta_0 T \leq \Delta; \\
        \nu \beta^{r+1}\ell_r \leq L_r;\\
        \frac{\nu\beta}{2} \geq \varepsilon; \\
        \nu^{p}\beta^{p r}\ell^{p}_{r-1}\frac{2}{\rho^{p-1}} \leq \sigma^{p}_r.
    \end{cases}
\end{equation*}
We will resolve the system of inequalities step by step. The first step is to set $\nu = \frac{2\varepsilon}{\beta}$.  For $r = 1$, we have 
\begin{eqnarray*}
    \nu^{p}\beta^{p}\ell_0^{p} \cdot \frac{2}{\rho^{p-1}} \leq \sigma^{p}_1\quad \Rightarrow\quad \rho = \min\left\{\left(\frac{\varepsilon}{\sigma_1}\right)^{\frac{p}{p-1}}\cdot 2^{\frac{p+1}{p-1}}\ell_0^{\frac{p}{p-1}}, 1\right\} .
\end{eqnarray*} 
To set  $\beta$, we need to find its specific value, which can satisfy the following constraints: for any $r \in \{2, \dots, q\} $ and  any $r' \in [q]$ 
\begin{equation*}
    \beta^{r+1} \leq \frac{L_r}{\nu \ell_r} = \frac{\beta L_r}{2\varepsilon\ell_{r}}, \quad \nu^{p}\beta^{p r}\ell^{p}_{r-1}\frac{2}{\rho^{p-1}}\overset{(\ast)}{\leq} \beta^{p(r-1)}\left(\frac{\sigma_1 \ell_{r-1}}{\ell_0}\right)^p \leq \sigma^{p}_r,
\end{equation*}
where in the inequality $(\ast)$ we used  $\frac{2}{\rho^{p-1}}\leq \frac{\sigma^p_1}{\nu^p\beta^p\ell_0^p}$. Therefore, we obtain 
\begin{equation*}
    \beta = \min_{{r' \in [q];}~{r \in \{2,\dots, q\}}}\min\left\{ \left(\frac{\ell_0\sigma_r}{\ell_{r-1}\sigma_1}\right)^{\frac{1}{r-1}}, \left(\frac{L_{r'}}{2\varepsilon\ell_{r'}}\right)^{\frac{1}{r'}}\right\}.
\end{equation*}
Then, assuming $T \geq 3$ and $\left(\frac{\varepsilon}{\sigma_1}\right)^{\frac{p}{p-1}}\cdot 2^{\frac{p+1}{p-1}}\ell_0^{\frac{p}{p-1}} \leq 1$, we get
\begin{eqnarray*}
    \frac{T-1}{2\rho} &=&\frac{1}{2\rho}\left(\left\lfloor\frac{\Delta\beta}{2\Delta_0\varepsilon}\right\rfloor -1 \right) \geq  \frac{\Delta \beta}{8\rho\Delta_0\varepsilon} \\
    &\geq& \left(\frac{\sigma_1}{2\varepsilon\ell_0}\right)^{\frac{p}{p-1}} \cdot \frac{\Delta}{4\Delta_0\varepsilon} \min_{{r' \in [q];}~{r \in \{2,\dots, q\}}}\min\left\{ \left(\frac{\ell_0\sigma_r}{\ell_{r-1}\sigma_1}\right)^{\frac{1}{r-1}}, \left(\frac{L_{r'}}{2\varepsilon\ell_{r'}}\right)^{\frac{1}{r'}}\right\}\\
    &\geq&\Omega(1)\cdot\frac{\Delta}{\varepsilon}\left(\frac{\sigma_1}{\varepsilon}\right)^{\frac{p}{p-1}}  \min\left\{ \min_{r \in \{2,\dots, q\}}\left(\frac{\sigma_r}{\sigma_1}\right)^{\frac{1}{r-1}}, \min_{r' \in [q]}\left(\frac{L_{r'}}{\varepsilon}\right)^{\frac{1}{r'}}\right\}.
\end{eqnarray*}
Since $h$ and $h^{\star}$ are functions of $T$ arguments, the dimension of the problem is equal to $T$. This concludes the proof. 
\end{proof}

\newpage
\section{Missing Proofs for Section~\ref{sec:NSGD_SOM_conv}}
\label{app:missing_proofs_NSGD_SOM}
In this section we study Normalized SGD with Hessian correction (NSGD-Hess) (see Algorithm~\ref{alg:NSGD_SOM}). 

\textbf{Auxiliary Lemmas.} To show the convergence guarantees for Algorithm~\ref{alg:NSGD_SOM}, we state and prove auxiliary lemmas. The first one is well-known \textit{Descent Lemma} for Normalized SGD. The proof for this lemma can be found in \cite{MomentumImprovesNSGD_Cutkosky_2020,Hubler2024clip_to_norm}, but for completeness we restate and reprove it. The second one is devoted to the bounds on \textit{error} term from \textit{Descent Lemma}. 

\begin{lemma}
\label{lem:descent_lemma}
    Let Assumptions~\ref{assum:lower_bounded} and~\ref{assum:l_smooth} hold. Then for any selection of stepsize $\gamma >0$ the iterates $\{x_t\}^{T}_{t=0}$ generated by Algorithm~\ref{alg:NSGD_SOM} satisfy
    \begin{equation}
        \gamma \sum^{T-1}_{t=0}\|\nabla F(x_t)\| + \Delta_{T} \leq \Delta_0 +2\gamma \sum^{T-1}_{t=0}\|\hat{e}_t\| + \frac{\gamma^2 L T}{2}.
    \end{equation}
    where functional gap is defined as $\Delta_t \eqdef F(x_t) -F_*$, error term is defined as $\hat{e}_t \eqdef g_t - \nabla F(x_t)$.
\end{lemma}
\begin{proof}
    According to the update rule for $x_t$ and Assumption~\ref{assum:l_smooth}, we have 
    \begin{eqnarray*}
    F(x_{t+1}) &\leq& F(x_t) +\la \nabla F(x_t), x_{t+1} -x_{t} \ra + \frac{L}{2}\|x_{t+1} - x_t\|^2 \\
    &=& F(x_t) - \gamma \left\la \nabla F(x_t), \frac{g_t}{\|g_t\|} \right\ra + \frac{\gamma^2 L}{2}\\
    &=& F(x_t) - \gamma \|g_t\| - \gamma \left\la \nabla F(x_t) - g_t, \frac{g_t}{\|g_t\|} \right\ra + \frac{\gamma^2 L}{2}\\
    &\leq& F(x_t) - \gamma \|g_t\| +\gamma \|\nabla F(x_t) - g_t\| + \frac{\gamma^2 L}{2}\\
    &\leq& F(x_t) - \gamma \|\nabla F(x_t)\| +2\gamma \|\nabla F(x_t) - g_t\| + \frac{\gamma^2 L}{2}.
\end{eqnarray*}
Using notation for the functional gap and the error term, we get
\begin{equation*}
    \gamma \|\nabla F(x_t)\| + \Delta_{t+1} \leq \Delta_t +2\gamma \|\hat{e}_t\| + \frac{\gamma^2 L}{2}.
\end{equation*}
Summing over $t$ from $0$ to $T-1$, we obtain 
\begin{equation}
    \gamma \sum^{T-1}_{t=0}\|\nabla F(x_t)\| + \Delta_{T} \leq \Delta_0 +2\gamma \sum^{T-1}_{t=0}\|\hat{e}_t\| + \frac{\gamma^2 L T}{2}.
\end{equation}
\end{proof}

Next, we derive the lemma for error control for the Hessian corrected momentum estimator (lines $4-7$ in Algorithm~\ref{alg:NSGD_SOM}). Similar lemma for the special case $p=2$ appeared previously in \citep{salehkaleybar-et-al22,Fatkhullin_SPGM_FND_2023}. For the case $p<2$, similar recursion was derived by \citet{Hubler2024clip_to_norm} for first-order momentum. Now we extend this idea to the case of second-order momentum. 
\begin{lemma}\label{lem:dev_bound}
    Let Assumptions~\ref{assum:l_smooth},~\ref{ass:p_BCM} and ~\ref{ass:p_BCM_hess} hold. Then for all $t \geq 0$, we have 
\begin{equation}\label{eq:dev_bound}
     \Exp{\norm{ \hat{e}_{t} }} \leq (1-\alpha)^{t} \Exp{\norm{\hat{e}_0}} + 2\alpha^{\frac{p-1}{p}}\sigma + 12\gamma(L+\sigma_h)\alpha^{-\nicefrac{1}{p}},
\end{equation}
where $\hat e_t : = g_t - \nabla F(x_t)$.
\end{lemma}
\begin{proof}
Define $e_t = \nabla f(x_t, \xi_t) - \nabla F(x_t)$, $ \hat S_{t+1} =  \nabla^2 f(\hat x_{t+1}, \hat \xi_{t+1}) \cdot (x_{t+1} - x_t)  - \nabla F(x_{t+1}) + \nabla F(x_t)$. We have $\Exp{e_t} = 0$, $\Exp{\norm{e_t}^p} \leq \sigma^p$, $\Exp{\hat S_t} = 0$, and 
\begin{eqnarray*}
\Exp{\norm{\hat S_t}^p} &=& \Exp{\norm{ \nabla F(x_t) - \nabla F(x_{t+1}) + \nabla^2 f(\hat x_{t+1}, \hat \xi_{t+1}) \cdot (x_{t+1} - x_t) }^p } \\
&\leq& 3 \Exp{\norm{ \nabla F(x_t) - \nabla F(x_{t+1}) }^p} + 3 \Exp{\norm{ \nabla^2 F(\hat x_{t+1}) \cdot (x_{t+1} - x_t) }^p }  \\
&& \qquad  + 3 \Exp{\norm{ \rb{ \nabla^2 f(\hat x_{t+1}, \hat \xi_{t+1}) - \nabla^2 F(\hat x_{t+1}) }  \cdot (x_{t+1} - x_t) }^p } \\
&\leq& 6 L^p \gamma^p + 3 \sigma_h^p \gamma^p  . 
\end{eqnarray*}

By the update rule for the gradient estimator: 
\begin{eqnarray}
  \hat{e}_{t} &=& {g}_{t} - \nabla F(x_{t}) =   (1- \alpha)  \hat{e}_{t-1} + \alpha  e_{t} + (1- \alpha) \hat S_t  \notag .
\end{eqnarray}
Unrolling the recursion, we have
    \begin{eqnarray}\label{eq:hat_et_unrolled}
    \hat{e}_{t} &= &  (1-\alpha)^{t} \hat{e}_0 + \alpha \sum_{j = 0}^{t-1} (1-\alpha)^{t-j-1}  e_{j+1}  + \sum_{j = 0}^{t-1} (1-\alpha)^{t-j} \hat S_{j+1} \,. \notag 
\end{eqnarray}
Then, taking the norm and the total expectation, we get 
\begin{equation}
\label{eq:dhsbhdsdsiwp}
    \Exp{\norm{ \hat{e}_{t} }} \leq(1-\alpha)^{t} \Exp{\norm{\hat{e}_0}} + \Exp{\norm{\alpha \sum_{j = 0}^{t-1} (1-\alpha)^{t-j-1}  e_{j+1}}}  + \Exp{\norm{\sum_{j = 0}^{t-1} (1-\alpha)^{t-j} \hat S_{j+1}}}.
\end{equation}
To continue the proof, we need to bound the last two terms from the previous inequality. By Jensen's inequality, we obtain
\begin{eqnarray}
    \Exp{\norm{\alpha \sum_{j = 0}^{t-1} (1-\alpha)^{t-j-1}  e_{j+1}}} &\leq& \left(\Exp{\norm{\alpha \sum_{j = 0}^{t-1} (1-\alpha)^{t-j-1}  e_{j+1}}^p}\right)^{\nicefrac{1}{p}} \notag\\
    &\overset{(\ast)}{\leq}& \left(2\alpha^p \sum_{j = 0}^{t-1} (1-\alpha)^{p(t-j-1)}  \Exp{\norm{e_{j+1}}^p}\right)^{\nicefrac{1}{p}} \notag\\
    &\overset{(\ast\ast)}{\leq}& \left(2\alpha^p \sum_{j = 0}^{t-1} (1-\alpha)^{p(t-j-1)}  \sigma^p\right)^{\nicefrac{1}{p}} \notag\\
    &\overset{(\ast\ast\ast)}{\leq}& 2 \left(\alpha^{p-1} \sigma^p\right)^{\nicefrac{1}{p}} = 2\alpha^{\frac{p-1}{p}}\sigma , \label{eq:bjhdvbfuioiapa}
\end{eqnarray}
where in $(\ast)$ we used Lemma~\ref{lem:von_Bahr_inequality}, in $(\ast\ast)$ we used $\Exp{\norm{e_{j+1}}^p} \leq \sigma^p$and  in $(\ast\ast\ast)$ we used the following inequality
\begin{equation*}
    \sum_{j = 0}^{t-1} (1-\alpha)^{p(t-j-1)} \leq \sum_{j = 0}^{t-1} (1-\alpha)^{t-j-1}   \leq \sum_{j = 0}^{\infty} (1-\alpha)^{j}  = \frac{1}{\alpha}. 
\end{equation*}
We bound the third term in the same way as we did for the second term:
\begin{eqnarray}
    \Exp{\norm{ \sum_{j = 0}^{t-1} (1-\alpha)^{t-j}  \hat S_{j+1}}} &\leq& \left(\Exp{\norm{\alpha \sum_{j = 0}^{t-1} (1-\alpha)^{t-j}  \hat S_{j+1}}^p}\right)^{\nicefrac{1}{p}} \notag\\
    &\overset{(\ast)}{\leq}& \left(2 \sum_{j = 0}^{t-1} (1-\alpha)^{p(t-j)}  \Exp{\norm{\hat S_{j+1}}^p}\right)^{\nicefrac{1}{p}} \notag\\
    &\overset{(\ast\ast)}{\leq}& \left(12 \sum_{j = 0}^{t-1} (1-\alpha)^{p(t-j)} \gamma^p (L^p +\sigma^p_h)\right)^{\nicefrac{1}{p}} \notag\\
    &\overset{(\ast\ast\ast)}{\leq}& 12 \left(\alpha^{-1} \gamma^p(L^p +\sigma^p_h))\right)^{\nicefrac{1}{p}} \leq 12(L+\sigma_h)\gamma\alpha^{-\nicefrac{1}{p}} , \label{eq:hdsbhfjbshbhf}
\end{eqnarray}
where in $(\ast)$ we used Lemma~\ref{lem:von_Bahr_inequality}, in $(\ast\ast)$ we used $\Exp{\norm{\hat S_{j+1}}^p} \leq 6L^p\gamma^p + 3\sigma^p\gamma^p$, and  in $(\ast\ast\ast)$ we used the following inequality
\begin{equation*}
    \sum_{j = 0}^{t-1} (1-\alpha)^{p(t-j)} \leq \sum_{j = 0}^{t-1} (1-\alpha)^{t-j}   \leq \sum_{j = 0}^{\infty} (1-\alpha)^{j}  = \alpha^{-1}. 
\end{equation*}

Plugging \eqref{eq:bjhdvbfuioiapa} and \eqref{eq:hdsbhfjbshbhf} into \eqref{eq:dhsbhdsdsiwp}, we obtain
\begin{eqnarray*}
    \Exp{\norm{ \hat{e}_{t} }} \leq (1-\alpha)^{t} \Exp{\norm{\hat{e}_0}} + 2\alpha^{\frac{p-1}{p}}\sigma + 12\gamma(L+\sigma_h)\alpha^{-\nicefrac{1}{p}},
\end{eqnarray*}
which concludes the proof. 
\end{proof}

\textbf{Main Theorem.} Now we are ready to state and to prove the main convergence theorem for Algorithm~\ref{alg:NSGD_SOM}. 
\begin{theorem}
\label{thm:nsgdm_hess_app}
    Let Assumptions~\ref{assum:lower_bounded},~\ref{assum:l_smooth},~\ref{ass:p_BCM} and ~\ref{ass:p_BCM_hess} hold, and stepsize $\gamma = \sqrt{\frac{\Delta \alpha^{\nicefrac{1}{p}}}{ (L + \sigma_h)  T}}$, momentum parameter $\alpha = \min\left\{1,\alpha_{\text{eff}}\right\}$, where $\alpha_{\text{eff}} = \max\left\{\left(\frac{\cE_0}{T\sigma}\right)^{\frac{p}{2p-1}}, \left(\frac{\Delta(L+\sigma_h)}{T\sigma^2}\right)^{\frac{p}{2p-1}}\right\}$. Then iterates $\{x_t\}_{t=0}^{T-1}$ of Algorithm~\ref{alg:NSGD_SOM} satisfy
    \begin{equation*}
        \frac{1}{T}\sum^{T-1}_{t=0}\Exp{\norm{\nabla F(x_t)}} = \cO\left(\sqrt{\frac{\Delta (L +\sigma_h)}{T}}  +\sigma \left(\frac{\Delta(L+\sigma_h)}{T\sigma^2}\right)^{\frac{p-1}{2p-1}} + \frac{\cE_0}{ T} + \sigma\left(\frac{\cE_0}{T\sigma}\right)^{\frac{p-1}{2p-1}}\right),
    \end{equation*}
    where $\cE_0$ is defined as some  upper bound on $\EE\left[\norm{g_0 - \nabla F(x_0)}\right]$
\end{theorem}
\begin{proof}
    Applying Lemma~\ref{lem:dev_bound}, we have 
    \begin{eqnarray*}
        \sum^{T-1}_{t=0}\Exp{\norm{ \hat{e}_{t} }} &\leq& \sum^{T-1}_{t=0}(1-\alpha)^{t} \Exp{\norm{\hat{e}_0}} + 2\alpha^{\frac{p-1}{p}}\sigma T + 12\gamma(L+\sigma_h)\alpha^{-\nicefrac{1}{p}}T\\
        &\leq& \frac{\Exp{\norm{\hat{e}_0}}}{\alpha} +  2\alpha^{\frac{p-1}{p}}\sigma T + 12\gamma(L+\sigma_h)\alpha^{-\nicefrac{1}{p}}T.
    \end{eqnarray*}
    Therefore, according to Lemma~\ref{lem:descent_lemma}, we obtain 
    \begin{eqnarray*}
        \frac{1}{T}\sum^{T-1}_{t=0}\EE\left[\|\nabla F(x_t)\|\right] &\leq& \frac{\Delta}{\gamma T} +\frac2T \sum^{T-1}_{t=0}\EE\left[\|\hat{e}_t\|\right] + \frac{\gamma L}{2}\\
        &\leq&  \frac{\Delta}{\gamma T}+ \frac{\gamma L}{2}  + \frac{2\Exp{\norm{\hat{e}_0}}}{\alpha T} +  4\alpha^{\frac{p-1}{p}}\sigma  + 24\gamma(L+\sigma_h)\alpha^{-\nicefrac{1}{p}}\\
        &\leq&  \frac{\Delta}{\gamma T}  + \frac{2\Exp{\norm{\hat{e}_0}}}{\alpha T} +  4\alpha^{\frac{p-1}{p}}\sigma  + 25\gamma(L+\sigma_h)\alpha^{-\nicefrac{1}{p}}\\
        &\leq& \cO\left(\sqrt{\frac{\Delta (L +\sigma_h)}{\alpha^{\nicefrac{1}{p}}T}} + \frac{\Exp{\norm{\hat{e}_0}}}{\alpha T} +\alpha^{\frac{p-1}{p}}\sigma\right),
    \end{eqnarray*}
    where in the last inequality we took $\gamma = \sqrt{\frac{\Delta \alpha^{\nicefrac{1}{p}}}{ (L + \sigma_h)  T}}$. Denoting $\EE\left[\norm{\hat{e}_0}\right] \leq \cE_0$ and taking momentum parameter $\alpha = \min\left\{1,\alpha_{\text{eff}}\right\}$, where $\alpha_{\text{eff}} = \max\left\{\left(\frac{\cE_0}{T\sigma}\right)^{\frac{p}{2p-1}}, \left(\frac{\Delta(L+\sigma_h)}{T\sigma^2}\right)^{\frac{p}{2p-1}}\right\}$, we have 
    \begin{eqnarray*}
        \frac{1}{T}\sum^{T-1}_{t=0}\EE\left[\|\nabla F(x_t)\|\right] &=& \cO\left(\sqrt{\frac{\Delta (L +\sigma_h)}{\alpha^{\nicefrac{1}{p}}T}} + \frac{\cE_0}{\alpha T} +\alpha^{\frac{p-1}{p}}\sigma\right)\\
        &=& \cO\left(\sqrt{\frac{\Delta (L +\sigma_h)}{T}} + \frac{\cE_0}{ T} +\sigma \left(\frac{\Delta(L+\sigma_h)}{T\sigma^2}\right)^{\frac{p-1}{2p-1}} + \sigma\left(\frac{\cE_0}{T\sigma}\right)^{\frac{p-1}{2p-1}}\right).
    \end{eqnarray*}
\end{proof}

Now we investigate how different choices of initial estimator $g_0$ affect the total sample complexity bound.
\begin{corollary}
Let all assumptions of Theorem~\ref{thm:nsgdm_hess_app} hold and the step-size and momentum parameters are set according to this theorem statement.
\begin{enumerate}
    \item If we set $g_0 =\nabla F(x_0)$, then $\cE_0 = 0$, and the total sample complexity of Algorithm~\ref{alg:NSGD_SOM} is equal to 
    \begin{eqnarray*}
        \cO \left(\frac{\Delta (\sm + \sigma_h)}{\varepsilon^2} + \frac{\Delta (\sm + \sigma_h)}{\varepsilon^2} \left(\frac{\sigma}{\varepsilon}\right)^{\frac{1}{p-1}} \right).
    \end{eqnarray*}
    \item If we set $g_0 =0$, then $\cE_0 = \sqrt{2\Delta L}$, and the total sample complexity of Algorithm~\ref{alg:NSGD_SOM} is equal to 
    \begin{equation*}
        \cO \left(\frac{\Delta (\sm + \sigma_h)}{\varepsilon^2} + \frac{\Delta (\sm + \sigma_h)}{\varepsilon^2} \left(\frac{\sigma}{\varepsilon}\right)^{\frac{1}{p-1}} + \frac{\sqrt{\Delta L} \sigma}{\varepsilon^2}\left(\frac{\sigma}{\varepsilon}\right)^{\frac{1}{p-1}}\right).
    \end{equation*}
    \item If we set $g_0 = \frac{1}{B_{\text{init}}}\sum^{B_{\text{init}}}_{j = 1}\nabla f(x_0, \xi_{0, j})$ with $B_{\text{init}} = \max\left\{1, \left(\frac{\sigma}{\varepsilon}\right)^{\frac{p}{p-1}}, \left(\frac{\sigma}{\varepsilon}\right)^{\frac{p}{2p-1}}\right\}$, then $\cE_0 =2\sigma/ {B_{\text{init}}^{\frac{p-1}{p}}}$, and the total sample complexity of Algorithm~\ref{alg:NSGD_SOM} is equal to
    \begin{equation*}
        \cO\left(\frac{\Delta (L +\sigma_h)}{\varepsilon^2} + \frac{\Delta (\sm + \sigma_h)}{\varepsilon^2} \left(\frac{\sigma}{\varepsilon}\right)^{\frac{1}{p-1}} + \frac{\sigma}{\varepsilon}\left(\frac{\sigma}{\varepsilon}\right)^{\frac{1}{p-1}}  + \left(\frac{\sigma}{\varepsilon}\right)^{\frac{p}{2p-1}}\right).
    \end{equation*}
\end{enumerate}
\end{corollary}
\begin{proof}
    The first case is ideal, when we can have access to the full gradient but once: set $g_0 = \nabla F(x_0)$, then $\hat{e}_0 = 0$ and $\cE_0 = 0$. Thus we have 
    \begin{eqnarray*}
        \frac{1}{T}\sum^{T-1}_{t=0}\EE\left[\|\nabla F(x_t)\|\right] &=& \cO\left(\sqrt{\frac{\Delta (L +\sigma_h)}{T}} + \sigma \left(\frac{\Delta(L+\sigma_h)}{T\sigma^2}\right)^{\frac{p-1}{2p-1}} \right).
    \end{eqnarray*}
    In other words, we can guarantee $\frac{1}{T}\sum^{T-1}_{t=0}\EE\left[\|\nabla F(x_t)\|\right] \leq \varepsilon$ after 
    \begin{equation*}
        \cO \left(\frac{\Delta (\sm + \sigma_h)}{\varepsilon^2} + \frac{\Delta (\sm + \sigma_h)}{\varepsilon^2} \left(\frac{\sigma}{\varepsilon}\right)^{\frac{1}{p-1}} \right).
    \end{equation*}
    iterations of Algorithm~\ref{alg:NSGD_SOM}. 

    The second case is that we select $g_0 $ as a zero vector. This choice implies that by smoothness of $F$, we have 
    \begin{equation*}
        \EE\left[\norm{\hat{e}_0}\right] = \norm{\nabla F(x_0)} \leq \sqrt{2L\Delta} = \cE_0.
    \end{equation*}
    Plugging the obtained value of $\cE_0$ into the result of Theorem~\ref{thm:nsgdm_hess}, we have 
    \begin{eqnarray*}
        \frac{1}{T}\sum^{T-1}_{t=0}\EE\left[\|\nabla F(x_t)\|\right] &=& \cO\left(\sqrt{\frac{\Delta (L +\sigma_h)}{T}} + \sigma \left(\frac{\Delta(L+\sigma_h)}{T\sigma^2}\right)^{\frac{p-1}{2p-1}}  + \frac{\sqrt{\Delta L}}{T} + \sigma \left(\frac{\sqrt{\Delta L}}{T\sigma}\right)^{\frac{p-1}{2p-1}}\right)\\
        &=& \cO\left(\sqrt{\frac{\Delta (L +\sigma_h)}{T}} + \sigma \left(\frac{\Delta(L+\sigma_h)}{T\sigma^2}\right)^{\frac{p-1}{2p-1}}   + \sigma \left(\frac{\sqrt{\Delta L}}{T\sigma}\right)^{\frac{p-1}{2p-1}}\right)
    \end{eqnarray*}
    which implies that the total sample complexity is 
    \begin{equation*}
        \cO \left(\frac{\Delta (\sm + \sigma_h)}{\varepsilon^2} + \frac{\Delta (\sm + \sigma_h)}{\varepsilon^2} \left(\frac{\sigma}{\varepsilon}\right)^{\frac{1}{p-1}} + \frac{\sqrt{\Delta L} \sigma}{\varepsilon^2}\left(\frac{\sigma}{\varepsilon}\right)^{\frac{1}{p-1}}\right).
    \end{equation*}

    The third case is that we set $g_0 = \frac{1}{B_{\text{init}}}\sum^{B_{\text{init}}}_{j = 1}\nabla f(x_0, \xi_{0, j})$. By Lemma~\ref{lem:von_Bahr_inequality}, we have 
    \begin{eqnarray*}
        \Exp{\norm{\hat{e}_0}} &\leq& \left(\Exp{\norm{\hat{e}_0}^p}\right)^{\nicefrac{1}{p}} = \left(\Exp{\norm{g_0 -\nabla F(x_0)}^p}\right)^{\nicefrac{1}{p}}\\
        &\leq& \frac{2}{B_{\text{init}}} \left(\sum^{B_{\text{init}}}_{j=0}\Exp{\norm{\nabla f(x_0,\xi_{0,j}) -\nabla F(x_0)}^p}\right)^{\nicefrac{1}{p}}\\
        &\leq&  \frac{2}{B_{\text{init}}} \left(\sum^{B_{\text{init}}}_{j=0} \sigma^p\right)^{\nicefrac{1}{p}} = \frac{2\sigma}{B_{\text{init}}^{\frac{p-1}{p}}} = \cE_0.
    \end{eqnarray*}
    Plugging the obtained value of $\cE_0$ into the result of Theorem~\ref{thm:nsgdm_hess}, we have 
    \begin{eqnarray*}
        \frac{1}{T}\sum^{T-1}_{t=0}\EE\left[\|\nabla F(x_t)\|\right] &=& \cO\left(\sqrt{\frac{\Delta (L +\sigma_h)}{T}} + \sigma \left(\frac{\Delta(L+\sigma_h)}{T\sigma^2}\right)^{\frac{p-1}{2p-1}}  + \frac{\sigma}{T B_{\text{init}}^{\frac{p-1}{p}}} + \sigma \left(\frac{1}{T B_{\text{init}}^{\frac{p-1}{p}}}\right)^{\frac{p-1}{2p-1}}\right),
    \end{eqnarray*}
    from which we have that the total sample complexity is 
    \begin{eqnarray*}
         \cO\left(\frac{\Delta (L +\sigma_h)}{\varepsilon^2} + \frac{\Delta (\sm + \sigma_h)}{\varepsilon^2} \left(\frac{\sigma}{\varepsilon}\right)^{\frac{1}{p-1}} + \frac{\sigma^2}{B_{\text{init}}^{\frac{p-1}{p}}\varepsilon^2}\left(\frac{\sigma}{\varepsilon}\right)^{\frac{1}{p-1}}  + \frac{\sigma}{\varepsilon B_{\text{init}}^{\frac{p-1}{p}}}  + B_{\text{init}}\right).
    \end{eqnarray*}
    Taking $B_{\text{init}} = \max\left\{1, \left(\frac{\sigma}{\varepsilon}\right)^{\frac{p}{p-1}}, \left(\frac{\sigma}{\varepsilon}\right)^{\frac{p}{2p-1}}\right\}$, we have 
    that the total sample complexity is 
    \begin{eqnarray*}
         \cO\left(\frac{\Delta (L +\sigma_h)}{\varepsilon^2} + \frac{\Delta (\sm + \sigma_h)}{\varepsilon^2} \left(\frac{\sigma}{\varepsilon}\right)^{\frac{1}{p-1}} + \frac{\sigma}{\varepsilon}\left(\frac{\sigma}{\varepsilon}\right)^{\frac{1}{p-1}}  + \left(\frac{\sigma}{\varepsilon}\right)^{\frac{p}{2p-1}}\right).
    \end{eqnarray*}
\end{proof}

\newpage
\section{Missing Proofs for Section~\ref{sec:high_probability_analysis}}
\label{app:missing_proofs_hp}

Now we start the high-probability convergence analysis for Algorithm~\ref{alg:NSGD_SOM_clipped}. We divide our analysis into two parts: Optimization part, where we prove descent lemma, and High-Probability part, where we use concentration inequality to bound several terms from Descent Lemma. Finally, combining results from both parts, we prove the main results of this section via induction. The idea is based on work of \cite{sadiev2023high} and \cite{UnboundedClippedNSGDM2023Liu}.

\subsection{Analysis: Optimization Part}
 We start with \textit{Descent Lemma}. 

\begin{lemma}
\label{lem:descent_lemma_clipped_sgd}
    Let Assumptions~\ref{assum:lower_bounded},~\ref{assum:l_smooth} hold. Then Algorithm~\ref{alg:NSGD_SOM_clipped} with stepsize $\gamma > 0$ and momentum parameter $\alpha \in (0,1)$  generates iterates $\left\{x_t\right\}_{t=0}^{T}$ satisfying  the following inequality
    \begin{eqnarray*}
        \gamma \sum^{T-1}_{t=0}\|\nabla F(x_t)\| + \Delta_{T} &\leq& \Delta_1 + \frac{\gamma^2 L T}{2} + \frac{3\gamma}{\alpha} \sqrt{L\Delta_1}\\
        && +  2\gamma\alpha \sum^{T-1}_{t=1}\left\|\sum^t_{j=1}(1-\alpha)^{t-j}\theta_j\right\| + 2\gamma\sum^{T-1}_{t=1}\left\|\sum^t_{j=1}(1-\alpha)^{t-j+1}\omega_j\right\|,
    \end{eqnarray*}
    for any $j \in [T-1]$ vectors $\theta_j$ and $\omega_j$ are defined as follows
\begin{eqnarray}
    \theta_j &\eqdef&  \texttt{clip}\left(\nabla f(x_j, \xi_j), \lambda\right) - \nabla F(x_j), \label{eq:definiton_theta}\\
    \omega_j &\eqdef& \texttt{clip}\left(\nabla^2 f(\hat{x}_t, \hat{\xi}_t)(x_t - x_{t-1}), \lambda_h \right) - \left(\nabla F(x_j) - \nabla F(x_{j-1})\right). \label{eq:definition_omega}
\end{eqnarray}
\end{lemma}

 \begin{proof}
First, we notice that Lemma~\ref{lem:descent_lemma} holds for Algorithm~\ref{alg:NSGD_SOM_clipped} and can be proven in the same way as for Algorithm~\ref{alg:NSGD_SOM}, since the gradient updates in the both methods are identical except for $x_1$,  and the update rule for momentum term does not play any role in the proof of Lemma~\ref{lem:descent_lemma}. Therefore, the iterates $\left\{x_t\right\}_{t=0}^{T}$ of Algorithm~\ref{alg:NSGD_SOM_clipped} satisfy
\begin{equation}
\label{eq:descent_lemma_2}
    \gamma \sum^{T-1}_{t=0}\|\nabla F(x_t)\| + \Delta_{T} \leq \Delta_1 +2\gamma \sum^{T-1}_{t=0}\|\hat{e}_t\| + \frac{\gamma^2 L T}{2},
\end{equation}
where $\Delta_0 = \Delta_1 $, because $x_0 = x_1$. Next we bound $\|\hat{e}_t\|$ in almost the same way as we did in Lemma~\ref{lem:dev_bound}. By update rule for momentum parameter, we have 
\begin{eqnarray*}
    \hat e_t &=& g_t - \nabla F(x_t)\\
    &=&  (1-\alpha)\left( g_{t-1} + \texttt{clip}\left(\nabla^2 f(\hat{x}_t, \hat{\xi}_t)(x_t - x_{t-1}), \lambda_h \right)\right)\\
    &&\quad  + \alpha \texttt{clip}\left(\nabla f(x_t, \xi_t), \lambda\right) - \nabla F(x_t)\\
    &\overset{\eqref{eq:definiton_theta}, \eqref{eq:definition_omega}}{=}& (1-\alpha)\hat e_{t-1} + \alpha \theta_t +(1-\alpha)\omega_t\\
    &=& (1-\alpha)^t \hat e_0 + \alpha\sum^{t}_{j=1}(1-\alpha)^{t-j}\theta_j+ \sum^t_{j=1}(1-\alpha)^{t-j+1}\omega_j.
\end{eqnarray*}
Taking the norm, we obtain the following bound 
\begin{eqnarray}
    \|\hat{e}_t\| &\leq& (1-\alpha)^t\|\hat{e}_0\| + \alpha \left\|\sum^t_{j=1}(1-\alpha)^{t-j}\theta_j\right\| + \left\|\sum^t_{j=1}(1-\alpha)^{t-j+1}\omega_j\right\| \notag\\
    &\leq&  (1-\alpha)^t \sqrt{2L\Delta_1} + \alpha \left\|\sum^t_{j=1}(1-\alpha)^{t-j}\theta_j\right\| + \left\|\sum^t_{j=1}(1-\alpha)^{t-j+1}\omega_j\right\|, \label{eq:error_norm_bound}
\end{eqnarray}
where in the last inequality we used the following chain of inequalities 
\begin{equation*}
    \|\hat{e}_0\| =  \|g_0 - \nabla F(x_0)\| =  \| \nabla F(x_0)\| =  \|\nabla F(x_1)\| \leq \sqrt{2L (F(x_1) - F_*)} \leq \sqrt{2L\Delta_1}. 
\end{equation*}
    Plugging \eqref{eq:error_norm_bound} into \eqref{eq:descent_lemma_2}, we acquire
    \begin{eqnarray*}
        \gamma \sum^{T-1}_{t=0}\|\nabla F(x_t)\| + \Delta_{T} &\leq& \Delta_1 +2\gamma \sum^{T-1}_{t=0}\|\hat{e}_t\| + \frac{\gamma^2 L T}{2}\\
    &\leq&  \Delta_1 + \frac{\gamma^2 L T}{2} + 2\gamma \sum^{T-1}_{t=0}(1-\alpha)^t \sqrt{2L\Delta_1}\\
    && +  2\gamma\alpha \sum^{T-1}_{t=1}\left\|\sum^t_{j=1}(1-\alpha)^{t-j}\theta_j\right\| + 2\gamma\sum^{T-1}_{t=1}\left\|\sum^t_{j=1}(1-\alpha)^{t-j+1}\omega_j\right\|\\
    &\leq& \Delta_1 + \frac{\gamma^2 L T}{2} + \frac{3\gamma}{\alpha} \sqrt{L\Delta_1}\\
    && +  2\gamma\alpha \sum^{T-1}_{t=1}\left\|\sum^t_{j=1}(1-\alpha)^{t-j}\theta_j\right\| + 2\gamma\sum^{T-1}_{t=1}\left\|\sum^t_{j=1}(1-\alpha)^{t-j+1}\omega_j\right\|,
\end{eqnarray*}
where in the last inequality we used 
\begin{equation*}
    2\gamma \sum^{T-1}_{t=0}(1-\alpha)^t \sqrt{2L\Delta_1} \leq  2\gamma \sum^{\infty}_{t=0}(1-\alpha)^t \sqrt{2L\Delta_1} \leq \frac{2\sqrt{2}\gamma}{1-(1-\alpha)}\sqrt{L\Delta_1} \leq \frac{3\gamma}{\alpha}\sqrt{L\Delta_1}.
\end{equation*}
\end{proof}

\subsection{Analysis: Statistical Part}

According to Lemma~\ref{lem:descent_lemma_clipped_sgd}, we have two new terms
\begin{equation*}
    \left\|\sum^t_{j=1}(1-\alpha)^{t-j}\theta_j\right\| \quad \text{and}\quad  \left\|\sum^t_{j=1}(1-\alpha)^{t-j+1}\omega_j\right\|.
\end{equation*}
To bound both of them, we use the same idea as in the work of \cite{gorbunov2020stochastic, sadiev2023high, UnboundedClippedNSGDM2023Liu}: introduce unbiased and biased parts of $\theta_t$ and $\omega_t$, i.e. for any $t \in [T]$
\begin{eqnarray}
    \theta_t &=& \theta_t^u+\theta_t^b,\quad \text{where} \notag\\
    \theta^u_t &\eqdef& \texttt{clip}\left(\nabla f(x_t, \xi_t), \lambda\right) - \EE_{\xi_t}\left[\texttt{clip}\left(\nabla f(x_t, \xi_t), \lambda\right)\right], \label{eq:theta_unbiased}\\
    \theta^b_t &\eqdef& \EE_{\xi_t}\left[\texttt{clip}\left(\nabla f(x_t, \xi_t), \lambda\right)\right] - \nabla F(x_t); \label{eq:theta_biased}\\
    \omega_t &=& \omega^u_t +\omega_t^b,\quad \text{where} \notag\\
    \omega^u_t &\eqdef& \texttt{clip}\left(\nabla^2 f(\hat{x}_t, \hat{\xi}_t)(x_t - x_{t-1}), \lambda_h \right) - \EE_{q_t,\xi_t}\left[\texttt{clip}\left(\nabla^2 f(\hat{x}_t, \hat{\xi}_t)(x_t - x_{t-1}), \lambda_h \right)\right], \label{eq:omega_unbiased}\\
    \omega^b_t &\eqdef& \EE_{q_t,\xi_t}\left[\texttt{clip}\left(\nabla^2 f(\hat{x}_t, \hat{\xi}_t)(x_t - x_{t-1}), \lambda_h \right)\right] - \left(\nabla F(x_j) - \nabla F(x_{j-1})\right). \label{eq:omega_biased}
\end{eqnarray}
Under Assumption~\ref{ass:p_BCM} for any $\lambda \geq 2\|\nabla F(x_t)\|$ Lemma~\ref{lem:bound_variance_magnitude} implies 
\begin{equation}
\label{eq:theta_u_theta_b_variance_and_magnitude}
    \EE_{\xi_t}\left[\norm{\theta^u_j}^2\right] \leq 18\lambda^{2-p}\sigma^p\quad \text{and}\quad \norm{\theta^b_t} \leq 2^p\lambda^{1-p}\sigma^p,\quad \text{for all } t \in [T].
\end{equation}
It is worth to mention we have already shown in Lemma~\ref{lem:descent_lemma} that under Assumptions~\ref{assum:l_smooth} and~\ref{ass:p_BCM_hess} the vector $ \hat S_{t} =  \nabla^2 f(\hat x_{t}, \hat \xi_{t}) \cdot (x_{t} - x_{t-1})  - \nabla F(x_{t}) + \nabla F(x_{t-1})$ has zero expectation $\EE_{q_t, \xi_t}\left[\hat S_t\right] = 0$, since 
\begin{eqnarray*}
    \EE_{q_t, \hat \xi_t}\left[\nabla^2 f(\hat{x}_t, \hat{\xi}_t)(x_t - x_{t-1})\right] &=& \EE_{ q_t}\left[\EE_{\hat{\xi}_t}\left[\nabla^2 f(\hat{x}_t, \hat{\xi}_t)\right](x_t - x_{t-1})\right]\\
    &=& \EE_{ q_t}\left[\nabla^2 F(\hat{x}_t)(x_t - x_{t-1})\right]\\
    &=&\int^1_0\nabla^2 F(qx_t+(1-q)x_{t-1})(x_t-x_{t-1}) dq\\
    &=& \nabla F(x_t) - \nabla F(x_{t-1}).
\end{eqnarray*}
The $p$-central moment of $\hat S_t$ is bounded, i.e. $\EE_{q_t,\xi_t}\left[\norm{\hat S_t}^p\right] \leq  6 L^p \gamma^p + 3 \sigma_h^p \gamma^p $. Then, according to Lemma~\ref{lem:bound_variance_magnitude}, for all $t \in [T]$ we have 
\begin{equation}
\label{eq:omega_u_omega_b_variance_and_magnitude}
    \EE_{q_t,\xi_t}\left[\norm{\omega^u_j}^2\right] \leq 18\lambda^{2-p}_h\gamma^p(6 L^p  + 3 \sigma_h^p)\quad \text{and}\quad \norm{\omega^b_t} \leq 2^p\lambda^{1-p}_h\gamma^p (6 L^p + 3 \sigma_h^p),
\end{equation}
for any $2\norm{\nabla F(x_{t}) - \nabla F(x_{t-1})} \leq 2L\gamma \leq \lambda_h $.

\vspace{0.5cm}

\begin{lemma}
\label{lem:technical_lemma_hp_1}
    Let Assumption~\ref{ass:p_BCM} hold. For any $\delta' \in (0,\nicefrac{1}{2}] $ and any $t \in [T]$, if clipping level satisfy 
    $$\lambda \geq \max\left\{2\|\nabla F(x_j)\|, \frac{\sigma}{\alpha^{\nicefrac{1}{p}}}\right\}, $$
    then with probability at least $1-2\delta'$ 
    \begin{eqnarray*}
        \left\|\sum\limits^t_{j=1}(1-\alpha)^{t-j}\theta_j\right\|  \leq 22\lambda\log\frac{2}{\delta'}.
    \end{eqnarray*}
\end{lemma}
\begin{proof} 
    We start with bounding $\left\|\sum^t_{j=1}(1-\alpha)^{t-j}\theta_j\right\|$. 
    
    \paragraph{Upper bound  for $\left\|\sum^t_{j=1}(1-\alpha)^{t-j}\theta_j\right\|$ } By \eqref{eq:theta_unbiased} and \eqref{eq:theta_biased}, we have 
    \begin{eqnarray*}
        \left\|\sum\limits^t_{j=1}(1-\alpha)^{t-j}\theta_j\right\| &\leq& \left\|\sum\limits^t_{j=1}(1-\alpha)^{t-j}\theta_j^u\right\| + \underbrace{\left\|\sum\limits^t_{j=1}(1-\alpha)^{t-j}\theta_j^b\right\|}_{\circledFour}.
    \end{eqnarray*}

    Denote $Y_j^t \eqdef \left\|(1-\alpha)^{t-j}\theta_j^u\right\|^2 - \EE_{\xi_j}\left[\left\|(1-\alpha)^{t-j}\theta_j^u\right\|^2\right]$, and $|X_j^t| \leq \norm{(1-\alpha)^{t-j}\theta_j^u}$ for any $j \in [t]$
\begin{equation*}
    V^t_j \eqdef \begin{cases}
        0,& \text{if } j = 0;\\
        \sign\left(\sum\limits^{j-1}_{i=1} V^t_i\right) \frac{\left\la \sum\limits^{j-1}_{i=1}(1-\alpha)^{t-i}\theta_i^u, (1-\alpha)^{t-j}\theta_j^u \right\ra}{\left\|\sum\limits^{j-1}_{i=1}(1-\alpha)^{t-i}\theta_i^u\right\|}, & \text{if } j\neq 0 \text{ and } \sum\limits^{j-1}_{i=1}(1-\alpha)^{t-i}\theta_i^u \neq 0;\\
        0, &  \text{if } j\neq 0 \text{ and } \sum\limits^{j-1}_{i=1}(1-\alpha)^{t-i}\theta_i^u = 0.
    \end{cases}
\end{equation*}
    Then to bound $\left\|\sum\limits^t_{j=1}(1-\alpha)^{t-j}\theta_j^u\right\|$ we use Lemma~\ref{lem:inequlity_from_cutkosky} and obtain 
    \begin{eqnarray*}
        \left\|\sum\limits^t_{j=1}(1-\alpha)^{t-j}\theta^u_j\right\| &\leq&  \left|\sum^t_{j=1} V^t_j\right| + \sqrt{\max_{j \in [t]}\left\|(1-\alpha)^{t-j}\theta_j^u\right\|^2 + \sum\limits^t_{j=1}\left\|(1-\alpha)^{t-j}\theta_j^u\right\|^2}\\
        &\leq& \left|\sum^t_{j=1} V^t_j\right| + \sqrt{2 \sum\limits^t_{j=1}\left\|(1-\alpha)^{t-j}\theta_j^u\right\|^2}\\
        &\leq& \Bigg|\underbrace{\sum^t_{j=1} V^t_j}_{\circledOne}\Bigg| + \Bigg(2\cdot\underbrace{\sum\limits^t_{j=1}Y_j^t}_{\circledTwo} + 2\cdot\underbrace{\sum^t_{j=1} \EE_{\xi_j}\left[\left\|(1-\alpha)^{t-j}\theta_j^u\right\|^2\right]}_{\circledThree}\Bigg)^{\nicefrac{1}{2}}.
    \end{eqnarray*}

\textbf{Upper bound for \circledOne.} The sequence $X_1^t, \dots, X_t^t$ is martingale difference sequence, since, by definition of $\theta^t_j$ and $X^t_j$, for all $j \in [t]$ we have  $\EE\left[X^t_j~|~X^t_{j-1},\dots, X^t_1\right] = \EE_{\xi_j}\left[X^t_j\right] = 0$. Also, according to Lemma~\ref{lem:inequlity_from_cutkosky}, we have  $|X_j^t| \leq \norm{(1-\alpha)^{t-j}\theta_j^u}$ for all $j \in [t]$. Using Lemma~\ref{lem:bound_variance_magnitude}, we obtain that
\begin{equation}
    |X_j^t| \leq \norm{(1-\alpha)^{t-j}\theta_j^u} \leq \norm{\theta_j^u} \leq 2\lambda, \quad \text{where } c_1 \eqdef 2\lambda.
\end{equation}
Denoting $\sigma^2_j \eqdef \EE\left[(X^t_j)^2~|~X^t_{j-1},\dots, X^t_1\right] =\EE_{\xi_j}\left[(X^t_j)^2\right] $, Lemma~\ref{lem:Bernstain_inequality} implies 
\begin{equation*}
    \PP\left\{\left|\circledOne\right| > b_1 \text{ and } \sum^t_{j=1}\sigma^2_j \leq G_1 \log \frac{2}{\delta'}\right\} \leq 2\exp\left(- \frac{b^2_1}{2G_1\log\frac{2}{\delta'} + \frac{2c_1b_1}{3} }\right) = \delta',
\end{equation*}
where the last identity is true, if we set $b_1 = \left(\frac{1}{3}c_1+\sqrt{\frac{1}{9}c_1 + 2G_1}\right)\log \frac{2}{\delta'}$. To define $G_1$, we need to bound $\sum^t_{j=1}\sigma^2_j$:
\begin{eqnarray}
    \sum^t_{j=1}\sigma^2_j &=& \sum^t_{j=1} \EE_{\xi_j}\left[(X^t_j)^2\right] \leq \sum^t_{j=1} \EE_{\xi_j}\left[\|(1-\alpha)^{t-j}\theta_j^u\|^2\right] = \sum^t_{j=1} (1-\alpha)^{2(t-j)}\EE_{\xi_j}\left[\|\theta_j^u\|^2\right] \notag\\
    &\overset{\eqref{eq:theta_u_theta_b_variance_and_magnitude}}{\leq}&  \sum^t_{j=1} (1-\alpha)^{2(t-j)} \cdot 18\lambda^{2-p}\sigma^p \leq \frac{18 \lambda^{2-p}\sigma^p}{1-(1-\alpha)^2}  \leq \frac{18 \lambda^{2-p}\sigma^p}{\alpha} = G_1. \label{eq:qwertyuaisja}
\end{eqnarray}
where we assumed $\|\nabla F(x_j)\| \leq \frac{\lambda}{2}$ for any $j \in [t]$.

\textbf{Upper bound for \circledTwo.} The sequence $Y_1^t, \dots, Y_t^t$ forms a martingale difference sequence, since the definition of $Y^t_j$ leads to  $\EE\left[Y^t_j~|~Y^t_{j-1},\dots, Y^t_1\right] = \EE_{\xi_j}\left[Y^t_j\right] = 0$ for all $j \in [t]$ . Also, according to Lemma~\ref{lem:bound_variance_magnitude}, we obtain that
\begin{equation}
\label{eq:bound_on_Y}
    |Y_j^t| \leq \norm{(1-\alpha)^{t-j}\theta_j^u}^2 + \E_{\xi_j}\left[\norm{(1-\alpha)^{t-j}\theta_j^u}^2\right] \leq 4\lambda^2 + 4 \lambda^2 = 8\lambda^2, 
\end{equation}
where we define $c_2 \eqdef 8\lambda^2$. Denoting  the conditional  variance of $Y^t_j$ as $\widetilde{\sigma}^2_j \eqdef  \EE\left[(Y^t_j)^2~|~Y^t_{j-1},\dots, Y^t_1\right] = \EE_{\xi_j}\left[(Y^t_j)^2\right] $, we can bound $\widetilde{\sigma}^2_j$ easily
\begin{eqnarray*}
    \widetilde{\sigma}^2_j  &\overset{\eqref{eq:bound_on_Y}}{\leq}& 8\lambda^2 \cdot \EE\left[\left|\left\|(1-\alpha)^{t-j}\theta_j^u\right\|^2 - \EE_{\xi_j}\left[\left\|(1-\alpha)^{t-j}\theta_j^u\right\|^2\right]\right|\right]\\
    &\leq& 16\lambda^2 \EE_{\xi_j}\left[\left\|(1-\alpha)^{t-j}\theta_j^u\right\|^2\right].
\end{eqnarray*}
Then, Lemma~\ref{lem:Bernstain_inequality} implies 
\begin{equation*}
    \PP\left\{\left|\circledTwo\right| > b_2 \text{ and } \sum^t_{j=1}\widetilde{\sigma}^2_j \leq G_2 \log \frac{2}{\delta'}\right\} \leq 2\exp\left(- \frac{b^2_2}{2G_2\log\frac{2}{\delta'} + \frac{2c_2b_2}{3} }\right) = \delta',
\end{equation*}
where the last identity is true, if we set $b_2 = \left(\frac{1}{3}c_2+\sqrt{\frac{1}{9}c_2 + 2G_2}\right)\log \frac{2}{\delta'}$. To define $G_2$, we need to bound $\sum^t_{j=1}\widetilde{\sigma}^2_j $:
\begin{eqnarray*}
    \sum^t_{j=1} \widetilde{\sigma}^2_j &\leq& 16\lambda^2 \sum^t_{j=1}\EE_{\xi_j}\left[\left\|(1-\alpha)^{t-j}\theta_j^u\right\|^2\right]\\
    &=& 16\lambda^2 \sum^t_{j=1}(1-\alpha)^{2(t-j)}\EE_{\xi_j}\left[\left\|\theta_j^u\right\|^2\right]\\
    &\overset{\eqref{eq:qwertyuaisja}}{\leq}& 16\lambda^2\cdot \frac{18\lambda^{2-p}\sigma^p}{\alpha}  = \frac{16\cdot 18 \lambda^{4-p}\sigma^p}{\alpha} = G_2,
\end{eqnarray*}
where we assumed $\|\nabla F(x_j)\| \leq \frac{\lambda}{2}$ for any $j \in [t]$.

\textbf{Upper bound for \circledThree.} Thanks to the proof of the bound on \circledThree~ (see \eqref{eq:qwertyuaisja}), we have already shown what we need: with probability $1$
\begin{equation*}
    \circledThree = \sum^t_{j=1} \EE_{\xi_j}\left[\left\|(1-\alpha)^{t-j}\theta_j^u\right\|^2\right] \overset{\eqref{eq:qwertyuaisja}}{\leq} \frac{18 \lambda^{2-p}\sigma^p}{\alpha}. 
\end{equation*}

\textbf{Upper bound on \circledFour.}  Assuming $\|\nabla F(x_j)\| \leq \frac{\lambda}{2}$ for any $j \in [t]$, with probability $1$ we have 
\begin{eqnarray*}
    \circledFour   = \left\|\sum\limits^t_{j=1} (1-\alpha)^{t-j}\theta^b_j\right\| \leq \sum^t_{j=1}(1-\alpha)^{t-j}\|\theta^b_j\| \overset{\eqref{eq:qwertyuaisja}}{\leq} 4\lambda^{1-p}\sigma^p\sum^t_{j=1}(1-\alpha)^{t-j} \leq \frac{4\lambda^{1-p}\sigma^p}{\alpha}.
\end{eqnarray*}

To sum up, we introduce event $E_{\circledOne, t}$ as follows 
\begin{equation}
\label{eq:def_event_one}
    E_{\circledOne, t} \eqdef \left\{\left|\circledOne\right| \leq b_1 \text{ or } \sum^t_{j=1}\sigma^2_j > G_1\log\frac{2}{\delta'}\right\},
\end{equation}
where $c_1 = 2\lambda$, $G_1 = \frac{18\lambda^{2-p}\sigma^p}{\alpha}$, $b_1 = \left(\frac{1}{3}c_1 + \sqrt{\frac{1}{9}c_1^2 + 2G_1}\right)\log\frac{2}{\delta'}$.
We have shown that $\PP\left\{E_{\circledOne, t}\right\} \geq 1-\delta'$. The bound on $b_1$ 
\begin{eqnarray*}
    b_1 &=& \left(\frac{1}{3}c_1 + \sqrt{\frac{1}{9}c_1^2 + 2G_1}\right)\log\frac{2}{\delta'} \leq \left(\frac{2}{3}c_1 + \sqrt{2G_1}\right)\log\frac{2}{\delta'}\\
    &=& \left(\frac{4}{3}\lambda + \sqrt{36\frac{\lambda^{2-p}\sigma^p}{\alpha}}\right)\log\frac{2}{\delta'} =  \lambda\left(\frac{4}{3}+6\sqrt{\frac{1}{\alpha}\left(\frac{\sigma}{\lambda}\right)^p}\right)\log\frac{2}{\delta'}.
\end{eqnarray*}

Also we define event $E_{\circledTwo, t}$ as follows 
\begin{equation}
\label{eq:def_event_two}
    E_{\circledTwo, t} = \left\{\left|\circledTwo\right| \leq b_2 \text{ or } \sum^t_{j=1}\widetilde{\sigma}^2_j > G_2\log\frac{2}{\delta'}\right\},
\end{equation}
where $c_2 = 8\lambda^2$, $G_2 = \frac{16\cdot 18 \lambda^{4-p}\sigma^p}{\alpha}$, $b_2 = \left(\frac{1}{3}c_2 + \sqrt{\frac{1}{9}c^2_2 + 2G_2}\right)\log\frac{2}{\delta'}$. We have shown that $\PP\left\{E_{\circledTwo, t}\right\} \geq 1-\delta'$. We adjust the bound on $b_2$:
\begin{eqnarray*}
    b_2 &=& \left(\frac{1}{3}c_2 + \sqrt{\frac{1}{9}c_2^2 + 2G_2}\right)\log\frac{2}{\delta'} \leq  \left(\frac{2}{3}c_2 + \sqrt{2G_2}\right)\log\frac{2}{\delta'}\\
    &=& \left(\frac{16}{3}\lambda^2 + \sqrt{\frac{16 \cdot 36\lambda^{4-p}\sigma^p}{\alpha}}\right)\log\frac{2}{\delta'} = \lambda^2\left(\frac{16}{3}+24\sqrt{\frac{1}{\alpha}\left(\frac{\sigma}{\lambda}\right)^p}\right)\log\frac{2}{\delta'}.
\end{eqnarray*}

Thus, the event $E_{\circledOne, t} \cap E_{\circledTwo, t}$ implies 
\begin{eqnarray*}
    \left\|\sum\limits^t_{j=1}(1-\alpha)^{t-j}\theta_j\right\| &\leq& \left\|\sum\limits^t_{j=1}(1-\alpha)^{t-j}\theta_j^u\right\| + \left\|\sum\limits^t_{j=1}(1-\alpha)^{t-j}\theta_j^b\right\|\\
    &\leq& \left|\circledOne\right| + \sqrt{2\cdot \circledTwo + 2\cdot \circledThree} + \circledFour\\
    &\leq& \lambda\left(\frac{4}{3}+6\sqrt{\frac{1}{\alpha}\left(\frac{\sigma}{\lambda}\right)^p}\right)\log\frac{2}{\delta'} + 4\lambda \cdot \frac{1}{\alpha}\left(\frac{\sigma}{\lambda}\right)^p\\
    && + \sqrt{2\lambda^2\left(\frac{16}{3}+24\sqrt{\frac{1}{\alpha}\left(\frac{\sigma}{\lambda}\right)^p}\right)\log\frac{2}{\delta'} + \frac{36 \lambda^{2-p}\sigma^p}{\alpha}}.
\end{eqnarray*}
Taking $\lambda \geq \frac{\sigma}{\alpha^{\nicefrac{1}{p}}}$, we have the event $E_{\circledOne, t} \cap E_{\circledTwo, t}$ implies 
\begin{eqnarray*}
    \left\|\sum\limits^t_{j=1}(1-\alpha)^{t-j}\theta_j\right\| &\leq& \lambda\left(\frac{4}{3}+ 6 + 4 + \sqrt{
    \frac{32}{3} + 84}\right)\log\frac{2}{\delta'} \leq 22\lambda\log\frac{2}{\delta'}.
\end{eqnarray*}
This concludes the proof.
\end{proof}

\begin{lemma}
\label{lem:technical_lemma_hp_2}
    Let Assumptions~\ref{assum:l_smooth}and ~\ref{ass:p_BCM_hess} hold. For any $\delta'' \in (0,\nicefrac{1}{2}] $ and any $t \in [T]$, if clipping level satisfy 
    $$ \lambda_h \geq \max\left\{2\gamma L, \frac{\gamma(L+\sigma_h)}{\alpha^{\nicefrac{1}{p}}}\right\}, $$
    and with probability at least $1-2\delta''$
    \begin{eqnarray*}
        \left\|\sum^t_{j=1}(1-\alpha)^{t-j+1}\omega_j\right\| \leq  59\lambda_h\log\frac{2}{\delta''}.
\end{eqnarray*}
\end{lemma}

\begin{proof}
 By \eqref{eq:omega_unbiased} and \eqref{eq:omega_biased}, we have 
    \begin{eqnarray}
        \left\|\sum\limits^t_{j=1}(1-\alpha)^{t-j+1}\omega_j\right\| &\leq& \left\|\sum\limits^t_{j=1}(1-\alpha)^{t-j+1}\omega_j^u\right\| + \underbrace{\left\|\sum\limits^t_{j=1}(1-\alpha)^{t-j+1}\omega_j^b\right\|}_{\circledEight}. \label{eq:eight_def}
    \end{eqnarray}
    Denote $Z_j^t \eqdef \left\|(1-\alpha)^{t-j+1}\omega_j^u\right\|^2 - \EE_{q_j, \xi_j}\left[\left\|(1-\alpha)^{t-j+1}\omega_j^u\right\|^2\right]$, and $|W_j^t| \leq \norm{(1-\alpha)^{t-j+1}\omega_j^u}$ for any $j \in [t]$
\begin{equation*}
    W^t_j \eqdef \begin{cases}
        0,& \text{if } j = 0;\\
        \sign\left(\sum\limits^{j-1}_{i=1} W^t_i\right) \frac{\left\la \sum\limits^{j-1}_{i=1}(1-\alpha)^{t-i+1}\omega_i^u, (1-\alpha)^{t-j+1}\omega_j^u \right\ra}{\left\|\sum\limits^{j-1}_{i=1}(1-\alpha)^{t-i+1}\omega_i^u\right\|}, & \text{if } j\neq 0 \text{ and } \sum\limits^{j-1}_{i=1}(1-\alpha)^{t-i+1}\omega_i^u \neq 0;\\
        0, &  \text{if } j\neq 0 \text{ and } \sum\limits^{j-1}_{i=1}(1-\alpha)^{t-i+1}\omega_i^u = 0.
    \end{cases}
\end{equation*}
    Then to bound $\left\|\sum\limits^t_{j=1}(1-\alpha)^{t-j+1}\omega_j^u\right\|$ we use Lemma~\ref{lem:inequlity_from_cutkosky} and obtain 
    \begin{eqnarray*}
        \left\|\sum\limits^t_{j=1}(1-\alpha)^{t-j}\omega^u_j\right\| &\leq&  \left|\sum^t_{j=1} W^t_j\right| + \sqrt{\max_{j \in [t]}\left\|(1-\alpha)^{t-j}\omega_j^u\right\|^2 + \sum\limits^t_{j=1}\left\|(1-\alpha)^{t-j}\omega_j^u\right\|^2}\\
        &\leq& \left|\sum^t_{j=1} W^t_j\right| + \sqrt{2 \sum\limits^t_{j=1}\left\|(1-\alpha)^{t-j+1}\omega_j^u\right\|^2}\\
        &\leq& \Bigg|\underbrace{\sum^t_{j=1} W^t_j}_{\circledFive}\Bigg| + \Bigg(2\cdot\underbrace{\sum\limits^t_{j=1}Z_j^t}_{\circledSix} + 2\cdot\underbrace{\sum^t_{j=1} \EE_{q_j,\xi_j}\left[\left\|(1-\alpha)^{t-j+1}\omega_j^u\right\|^2\right]}_{\circledSeven}\Bigg)^{\nicefrac{1}{2}}.
    \end{eqnarray*}

\textbf{Upper bound for \circledFive.} The sequence $W_1^t, \dots, W_t^t$ is martingale difference sequence, since, by definition of $\omega^t_j$ and $W^t_j$, for all $j \in [t]$ we have  $\EE\left[W^t_j~|~W^t_{j-1},\dots, W^t_1\right] = \EE_{q_j, \xi_j}\left[W^t_j\right] = 0$. Also, according to Lemma~\ref{lem:inequlity_from_cutkosky}, we have  $|W_j^t| \leq \norm{(1-\alpha)^{t-j}\omega_j^u}$ for all $j \in [t]$. Using Lemma~\ref{lem:bound_variance_magnitude}, we obtain that
\begin{equation}
\label{eq:mdskcmsdl}
    |W_j^t| \leq \norm{(1-\alpha)^{t-j}\omega_j^u} \leq \norm{\omega_j^u} \leq 2\lambda_h, \quad \text{where } c_3 \eqdef 2\lambda_h.
\end{equation}
Denoting $\sigma^2_j \eqdef \EE\left[(W^t_j)^2~|~W^t_{j-1},\dots, W^t_1\right] =\EE_{\xi_j}\left[(W^t_j)^2\right] $, Lemma~\ref{lem:Bernstain_inequality} implies 
\begin{equation*}
    \PP\left\{\left|\circledFive\right| > b_3 \text{ and } \sum^t_{j=1}\Bar\sigma^2_j \leq G_3 \log \frac{2}{\delta''}\right\} \leq 2\exp\left(- \frac{b^2_3}{2G_3\log\frac{2}{\delta''} + \frac{2c_3b_3}{3} }\right) = \delta'',
\end{equation*}
where the last identity is true, if we set $b_3 = \left(\frac{1}{3}c_3+\sqrt{\frac{1}{9}c_3 + 2G_3}\right)\log \frac{2}{\delta''}$. To define $G_3$, we need to bound $\sum^t_{j=1}\Bar \sigma^2_j$:
\begin{eqnarray}
    \sum^t_{j=1}\Bar{\sigma}^2_j &=& \sum^t_{j=1} \EE_{q_j,\xi_j}\left[(W^t_j)^2\right] 
    \overset{\eqref{eq:mdskcmsdl}}{\leq}  \sum^t_{j=1} (1-\alpha)^{2(t-j+1)}\EE_{q_j,\xi_j}\left[\|\omega_j^u\|^2\right] \notag \\
    &\overset{\eqref{eq:omega_u_omega_b_variance_and_magnitude}}{\leq}&  \sum^t_{j=1} (1-\alpha)^{2(t-j+1)} \cdot 18\lambda^{2-p}_h\left(6L^p+3\sigma^p_h\right)\gamma^p
    \leq \frac{18 \lambda^{2-p}_h(6L^p+3\sigma^p_h)\gamma^p}{1-(1-\alpha)^2} \notag\\
    &=& \frac{18 \lambda^{2-p}_h\left(6L^p+3\sigma^p_h\right)\gamma^p}{\alpha(2-\alpha)} \leq \frac{18 \lambda^{2-p}_h\left(6L^p+3\sigma^p_h\right)\gamma^p}{\alpha} = G_3. \label{eq:mjkcniwecuidhhjs}
\end{eqnarray}

\textbf{Upper bound for \circledSix.} The sequence $Z_1^t, \dots, Z_t^t$ forms a martingale difference sequence, since the definition of $Z^t_j$ leads to  $\EE\left[Z^t_j~|~Z^t_{j-1},\dots, Z^t_1\right] = \EE_{q_j,\xi_j}\left[Z^t_j\right] = 0$ for all $j \in [t]$ . Also, according to Lemma~\ref{lem:bound_variance_magnitude}, we obtain that
\begin{equation}
\label{eq:bound_on_Z}
    |Z_j^t| \leq \norm{(1-\alpha)^{t-j+1}\omega_j^u}^2 + \E_{q_j, \xi_j}\left[\norm{(1-\alpha)^{t-j+1}\omega_j^u}^2\right] \leq 4\lambda^2_h + 4 \lambda^2_h = 8\lambda^2_h, 
\end{equation}
where we define $c_4 \eqdef 8\lambda^2_h$. Denoting  the conditional  variance of $Z^t_j$ as $\widehat{\sigma}^2_j \eqdef  \EE\left[(Z^t_j)^2~|~Z^t_{j-1},\dots, Z^t_1\right] = \EE_{q_j, \xi_j}\left[(Z^t_j)^2\right] $, we can bound $\widehat{\sigma}^2_j$ easily
\begin{eqnarray}
    \widehat{\sigma}^2_j  &\overset{\eqref{eq:bound_on_Z}}{\leq}& 8\lambda^2_h \cdot \EE_{q_j, \xi_j}\left[\left|\left\|(1-\alpha)^{t-j+1}\omega_j^u\right\|^2 - \EE_{q_j, \xi_j}\left[\left\|(1-\alpha)^{t-j+1}\omega_j^u\right\|^2\right]\right|\right] \notag\\
    &\leq& 16\lambda^2_h \EE_{q_j, \xi_j}\left[\left\|(1-\alpha)^{t-j+1}\omega_j^u\right\|^2\right]. \label{eq:lsmclksnciooeiw}
\end{eqnarray}
Then, Lemma~\ref{lem:Bernstain_inequality} implies 
\begin{equation*}
    \PP\left\{\left|\circledSix\right| > b_4 \text{ and } \sum^t_{j=1} \widehat{\sigma}^2_j \leq G_4 \log \frac{2}{\delta'}\right\} \leq 2\exp\left(-\frac{b^2_4}{2G_4\log\frac{2}{\delta'} +\frac{2c_4b_4}{3}}\right) = \delta'',
\end{equation*}
where the last identity is true, if we set $b_4 = \left(\frac{1}{3}c_4+\sqrt{\frac{1}{9}c_4 + 2G_4}\right)\log \frac{2}{\delta''}$. To define $G_4$, we need to bound $\sum^t_{j=1}\widehat{\sigma}^2_j $:
\begin{eqnarray*}
    \sum^t_{j=1} \widehat{\sigma}^2_j &\overset{\eqref{eq:lsmclksnciooeiw}}{\leq}& 16\lambda^2_h \sum^t_{j=1}\EE_{q_j, \xi_j}\left[\left\|(1-\alpha)^{t-j+1}\omega_j^u\right\|^2\right] = 16\lambda^2_h \sum^t_{j=1}(1-\alpha)^{2(t-j+1)}\EE_{q_j, \xi_j}\left[\left\|\omega_j^u\right\|^2\right]\\
    &\leq& 16\lambda^2_h\cdot \frac{18\lambda_h^{2-p}\left(6L^p+3\sigma^p_h\right)\gamma^p}{\alpha}  = \frac{16\cdot 18 \lambda^{4-p}_h(6L^p+3\sigma^p_h)\gamma^p}{\alpha} = G_4.
\end{eqnarray*}

\textbf{Upper bound for \circledSeven.} Thanks to the proof of the bound on \circledFour~ (see \eqref{eq:mjkcniwecuidhhjs}), we have already shown what we need: with probability $1$
\begin{equation*}
    \circledSeven = \sum^t_{j=1} \EE_{\xi_j}\left[\left\|(1-\alpha)^{t-j+1}\omega_j^u\right\|^2\right] \overset{\eqref{eq:mjkcniwecuidhhjs}}{\leq} \frac{18 \lambda^{2-p}(6L^p+3\sigma^p)\gamma^p}{\alpha}. 
\end{equation*}

\textbf{Upper bound on \circledEight.}  By definition of $\omega^b_j$, with probability $1$ we have 
\begin{eqnarray*}
    \circledEight &\overset{\eqref{eq:eight_def}}{=}& \left\|\sum\limits^t_{j=1} (1-\alpha)^{t-j+1}\omega^b_j\right\| \leq \sum^t_{j=1}(1-\alpha)^{t-j+1}\|\omega^b_j\|\\
    &\overset{\eqref{eq:omega_u_omega_b_variance_and_magnitude}}{\leq}& 4\lambda_h^{1-p}\left(6L^p +3\sigma^p_h\right)\gamma^p\sum^t_{j=1}(1-\alpha)^{t-j+1} \leq \frac{4\lambda_h^{1-p}\left(6L^p+3\sigma^p_h\right)\gamma^p}{\alpha}.
\end{eqnarray*}

To sum up, we introduce event $E_{\circledFive, t}$ as follows 
\begin{equation}
\label{eq:def_event_5}
    E_{\circledFive, t} = \left\{\left|\circledFive\right| \leq b_3 \text{ or } \sum^t_{j=1}\Bar{\sigma}^2_j > G_3\log\frac{2}{\delta''}\right\},
\end{equation}
where $c_3 = 2\lambda_h$, $G_3= \frac{18\lambda^{2-p}_h\left(6L^p+3\sigma^p_h\right)\gamma^p}{\alpha}$, $b_3 = \left(\frac{1}{3}c_3 + \sqrt{\frac{1}{9}c_3^2 + 2G_3}\right)\log\frac{2}{\delta''}$.
We have shown that $\PP\left\{E_{\circledFive, t}\right\} \geq 1-\delta''$. We adjust the bound on $b_3$:
\begin{eqnarray*}
    b_3 &=& \left(\frac{1}{3}c_3 + \sqrt{\frac{1}{9}c_3^2 + 2G_3}\right)\log\frac{2}{\delta'} \leq \left(\frac{2}{3}c_3 + \sqrt{2G_3}\right)\log\frac{2}{\delta''}\\
    &=& \left(\frac{4}{3}\lambda_h + \sqrt{36\frac{\lambda^{2-p}_h\left(6L^p + 3\sigma^p_h\right)\gamma^p}{\alpha}}\right)\log\frac{2}{\delta''}\\
    &=&  \lambda_h\left(\frac{4}{3}+6\sqrt{\frac{(6L^p+3\sigma^p_h)\gamma^p}{\alpha\lambda^p}}\right)\log\frac{2}{\delta''}.
\end{eqnarray*}

Also we introduce event $E_{\circledSix, t}$ as follows 
\begin{equation}
\label{eq:def_event_6}
    E_{\circledSix, t} = \left\{\left|\circledSix\right| \leq b_4 \text{ or } \sum^t_{j=1}\widehat{\sigma}^2_j > G_4\log\frac{2}{\delta''}\right\},
\end{equation}
where $c_4 = 8\lambda^2_h$, $G_4 = \frac{16\cdot 18 \lambda^{4-p}_h\left(6L^p+3\sigma^p_h\right)\gamma^p}{\alpha}$, $b_4 = \left(\frac{1}{3}c_4 + \sqrt{\frac{1}{9}c^2_4 + 2G_4}\right)\log\frac{2}{\delta''}$. We have shown that $\PP\left\{E_{\circledSix, t}\right\} \geq 1-\delta''$. We adjust the bound on $b_4$:
\begin{eqnarray*}
    b_4 &=& \left(\frac{1}{3}c_4 + \sqrt{\frac{1}{9}c_4^2 + 2G_4}\right)\log\frac{2}{\delta'} \leq  \left(\frac{2}{3}c_4 + \sqrt{2G_4}\right)\log\frac{2}{\delta''}\\
    &=& \left(\frac{16}{3}\lambda^2_h + \sqrt{\frac{16 \cdot 36\lambda_h^{4-p}\left(6L^p+3\sigma^p_h\right)\gamma^p}{\alpha}}\right)\log\frac{2}{\delta''} \\
    &=& \lambda^2_h\left(\frac{16}{3}+24\sqrt{\frac{\left(6L^p+3\sigma^p_h\right)\gamma^p}{\alpha\lambda^p_h}}\right)\log\frac{2}{\delta''}.
\end{eqnarray*}

Thus, the event $E_{\circledFive, t} \cap E_{\circledSix, t}$ implies 
\begin{eqnarray*}
    \left\|\sum\limits^t_{j=1}(1-\alpha)^{t-j+1}\omega_j\right\| &\leq& \left\|\sum\limits^t_{j=1}(1-\alpha)^{t-j+1}\omega_j^u\right\| + \left\|\sum\limits^t_{j=1}(1-\alpha)^{t-j+1}\omega_j^b\right\|\\
    &\leq& \left|\circledFive\right| + \sqrt{2\cdot \circledSix + 2\cdot \circledSeven} + \circledEight\\
    &\leq& \lambda_h\left(\frac{4}{3}+6\sqrt{\frac{\left(6L^p+3\sigma^p_h\right)\gamma^p}{\alpha\lambda^p_h}}\right)\log\frac{2}{\delta''} + 4\lambda_h \cdot \frac{\left(6L^p+3\sigma^p_h\right)\gamma^p}{\alpha\lambda^p_h}\\
    && + \sqrt{2\lambda^2_h\left(\frac{16}{3}+24\sqrt{\frac{\left(6L^p+3\sigma^p_h\right)\gamma^p}{\alpha\lambda^p_h}}\right)\log\frac{2}{\delta''} +36 \lambda^{2}_h\frac{(6L^p+3\sigma^p_h)}{\alpha\lambda^p_h}}.
\end{eqnarray*}
Assuming that $\lambda_h \geq \frac{\gamma(L+\sigma_h)}{\alpha^{\nicefrac{1}{p}}}$, we have the event $E_{\circledFive, t} \cap E_{\circledSix, t}$ implies 
\begin{eqnarray*}
    \left\|\sum\limits^t_{j=1}(1-\alpha)^{t-j+1}\omega_j\right\| &\leq& \lambda_h\left(\frac{4}{3}+ 6\sqrt{6} + 24 + \sqrt{
    \frac{32}{3} + 48\sqrt{6}+216}\right)\log\frac{2}{\delta''}\\
    &\leq& 59\lambda_h\log\frac{2}{\delta''}.
\end{eqnarray*}

\end{proof}

\paragraph{Main Theorem} Now we are ready to state and prove the main results of this section: high-probability convergence guarantees for Algorithm~\ref{alg:NSGD_SOM_clipped}. 

\begin{theorem}
\label{thm:main_theorem_hp_app}
Let Assumptions~\ref{assum:lower_bounded},~\ref{assum:l_smooth},~\ref{ass:p_BCM} and ~\ref{ass:p_BCM_hess}  hold, and $\Delta_1  = \Delta_0 = F(x_0) - F_*$ . Then if the parameters of Algorithm~\ref{alg:NSGD_SOM_clipped}: stepsize
\begin{equation}
\label{eq:stepsize_hp_part}
\small{
    \gamma \leq \min\Bigg\{\frac{1}{2}\sqrt{\frac{\Delta_1}{LT}}, \frac{\alpha}{12}\sqrt{\frac{\Delta_1}{L}}, \frac{1}{1408\alpha T \log \frac{8T}{\delta}}\sqrt{\frac{\Delta_1}{L}},\frac{\Delta_1}{352\sigma \alpha^{\frac{p-1}{p}}T\log\frac{8T}{\delta}}, \sqrt{\frac{\Delta_1 \alpha^{\nicefrac{1}{p}}}{968(L+\sigma_h)T\log\frac{8T}{\delta}}}\Bigg\},}
\end{equation}
momentum parameter and clipping levels  
\begin{equation}
\label{eq:momentum_clipping_level}
    \alpha = \frac{1}{T^{\frac{p}{2p-1}}},\quad \lambda = \max\left\{4\sqrt{L\Delta_1}, \frac{\sigma}{\alpha^{\nicefrac{1}{p}}}\right\}, \quad \lambda_h = \frac{2\gamma(L+\sigma_h)}{\alpha^{\nicefrac{1}{p}}}
\end{equation}
for some $T \geq 1$ and $\delta \in (0,1]$ such that $\log\frac{8T}{\delta} \geq 1$. Then after $T$ iterations of Algorithm~\ref{alg:NSGD_SOM_clipped} the iterates with probability at least $1-\delta$ satisfy 
\begin{equation}
    \frac{1}{T}\sum^{T-1}_{t=0}\norm{\nabla F(x_t)} \leq \frac{2\Delta_1}{\gamma T}.
\end{equation}
Moreover, if we set \eqref{eq:stepsize_hp_part} as identity, then the iterates generated by Algorithm~\ref{alg:NSGD_SOM_clipped} after $T$ iterations with probability at least $1-\delta$ satisfy 
\begin{equation}
    \frac{1}{T}\sum^{T-1}_{t=0}\norm{\nabla F(x_t)} = \cO\left(\frac{\max\left\{\sqrt{\Delta_1(L+\sigma_h)}, \sigma\right\}}{T^{\frac{p-1}{2p-1}}}\log \frac{T}{\delta}\right),
\end{equation}
implying that to achieve $\frac{1}{T}\sum^{T-1}_{t=0}\norm{\nabla F(x_t)} \leq \varepsilon$ with probability at least $1-\delta$ Algorithm~\ref{alg:NSGD_SOM_clipped} requires 
\begin{equation}
\small
    T = \cO\left(\left(\frac{\max\left\{\sqrt{\Delta_1(L+\sigma_h)}, \sigma\right\}}{\varepsilon}\right)^{\frac{2p-1}{p-1}}\log^{\frac{2p-1}{p-1}}\left(\frac{1}{\delta}\left(\frac{\max\left\{\sqrt{\Delta_1(L+\sigma_h)}, \sigma\right\}}{\varepsilon}\right)^{\frac{2p-1}{p-1}}\right)\right)
\end{equation}
iterations/samples. 
\end{theorem}
\begin{proof} We remind in the proofs of Lemma~\ref{lem:technical_lemma_hp_1} and Lemma~\ref{lem:technical_lemma_hp_2} we have shown 
\begin{eqnarray*}
    E_{\circledOne, t} \overset{\eqref{eq:def_event_one}}{=} \left\{\left|\circledOne\right| \leq b_1 \text{ or } \sum^t_{j=1}\sigma^2_j > G_1\log\frac{2}{\delta'}\right\}, &&  E_{\circledTwo, t} \overset{\eqref{eq:def_event_two}}{=} \left\{\left|\circledTwo\right| \leq b_2 \text{ or } \sum^t_{j=1}\widetilde{\sigma}^2_j > G_2\log\frac{2}{\delta'}\right\},\\
     E_{\circledFive, t} \overset{\eqref{eq:def_event_5}}{=} \left\{\left|\circledFive\right| \leq b_3 \text{ or } \sum^t_{j=1}\Bar{\sigma}^2_j > G_3\log\frac{2}{\delta''}\right\}, &&
     E_{\circledSix, t} \overset{\eqref{eq:def_event_6}}{=}\left\{\left|\circledSix\right| \leq b_4 \text{ or } \sum^t_{j=1}\widehat{\sigma}^2_j > G_4\log\frac{2}{\delta''}\right\},
\end{eqnarray*}
where 
$$c_1 = 2\lambda,~ G_1 = \frac{18\lambda^{2-p}\sigma^p}{\alpha}, ~b_1 \leq \lambda\left(\frac{4}{3}+6\sqrt{\frac{1}{\alpha}\left(\frac{\sigma}{\lambda}\right)^p}\right)\log\frac{2}{\delta'},$$
$$ c_2 = 8\lambda^2,~ G_2 = \frac{16\cdot 18 \lambda^{4-p}\sigma^p}{\alpha},~ b_2 \leq \lambda^2\left(\frac{16}{3}+24\sqrt{\frac{1}{\alpha}\left(\frac{\sigma}{\lambda}\right)^p}\right)\log\frac{2}{\delta'},$$
$$c_3 = 2\lambda_h,~ G_3= \frac{18\lambda^{2-p}_h\left(6L^p+3\sigma^p_h\right)\gamma^p}{\alpha},~ b_3 \leq \lambda_h\left(\frac{4}{3}+6\sqrt{\frac{(6L^p+3\sigma^p_h)\gamma^p}{\alpha\lambda^p}}\right)\log\frac{2}{\delta''},$$
$$c_4 =  8\lambda^2_h,~ G_4 = \frac{16\cdot 18 \lambda^{4-p}_h\left(6L^p+3\sigma^p_h\right)\gamma^p}{\alpha},~ b_4 \leq \lambda^2_h\left(\frac{16}{3}+24\sqrt{\frac{\left(6L^p+3\sigma^p_h\right)\gamma^p}{\alpha\lambda^p_h}}\right)\log\frac{2}{\delta''}. $$
Moreover, according to Lemma~\ref{lem:technical_lemma_hp_1} and Lemma~\ref{lem:technical_lemma_hp_2}, we have  
$$\PP\left\{E_{\circledOne, t}\right\} \geq 1-\delta',~~ \PP\left\{E_{\circledTwo, t}\right\} \geq 1-\delta',~~ \PP\left\{E_{\circledFive, t}\right\} \geq 1-\delta'',~~\PP\left\{E_{\circledSix, t}\right\} \geq 1-\delta''.$$

The idea of the proof is based on technique form work of \cite{gorbunov2020stochastic, sadiev2023high, UnboundedClippedNSGDM2023Liu}: via  mathematical induction we plan to show that for any $\tau \in \{0,1,\dots,T-1\}$ the event 
    \begin{equation*}
        G_{\tau} = E_{\tau}\cap E_{1,\tau} \cap E_{2,\tau} \cap E_{3,\tau} \cap E_{4,\tau}
    \end{equation*}
    holds with probability at least $1-\frac{\tau\delta}{T}$, where event $E_{\tau}$ is defined as 
    \begin{equation*}
        E_{\tau} = \left\{\gamma\sum^t_{s=0}\|\nabla F(x_s)\| + \Delta_{t+1} \leq 2\Delta_1,\quad \forall~t\leq \tau\right\},
    \end{equation*}
    and 
    \begin{equation*}
        E_{1,\tau} = \bigcap^{\tau}_{t=1} E_{\circledOne, t},\quad E_{2,\tau} = \bigcap^{\tau}_{t=1} E_{\circledTwo, t},\quad E_{3,\tau} = \bigcap^{\tau}_{t=1} E_{\circledFive, t},\quad E_{4,\tau} = \bigcap^{\tau}_{t=1} E_{\circledSix, t}.
    \end{equation*}

\textbf{Basis of induction.} Obviously, we have  $G_0 = E_0 = \left\{\Delta_0 = \Delta_1 \leq 2\Delta_1\right\}$ holds with probability $1 - \frac{\tau\delta}{T} \overset{\tau = 0}{=} 1$. 

\textbf{Step of induction.} Assume that the induction hypothesis is true for $\tau-1$:  $\PP\left\{G_{\tau-1}\right\} \geq 1 - \nicefrac{(\tau-1)\delta}{T}$. One needs to prove $\PP\left\{G_{\tau}\right\} \geq 1 - \nicefrac{\tau\delta}{T}$. We want to mention that 
\begin{eqnarray*}
    E_{\tau-1} \cap E_{1,\tau} \cap E_{2,\tau} \cap E_{3,\tau} \cap E_{4,\tau} = G_{\tau-1} \cap E_{\circledOne,\tau} \cap E_{\circledTwo,\tau} \cap E_{\circledFive,\tau} \cap E_{\circledSix,\tau}. 
\end{eqnarray*}
Then, we have 
\begin{eqnarray*}
    \PP\left\{E_{\tau-1} \cap E_{1,\tau} \cap E_{2,\tau} \cap E_{3,\tau} \cap E_{4,\tau}\right\} &=& \PP\left\{G_{\tau-1} \cap E_{\circledOne,\tau} \cap E_{\circledTwo,\tau} \cap E_{\circledFive,\tau} \cap E_{\circledSix,\tau}\right\}\\
    &=& 1- \PP\left\{\overline{G_{\tau-1} \cap E_{\circledOne,\tau} \cap E_{\circledTwo,\tau} \cap E_{\circledFive,\tau} \cap E_{\circledSix,\tau}}\right\}\\
    &\geq& 1- \PP\left\{\overline{G_{\tau-1}}\right\} - \PP\left\{\overline{E_{\circledOne,\tau}}\right\}- \PP\left\{\overline{E_{\circledTwo,\tau}}\right\}\\
    &&- \PP\left\{\overline{ E_{\circledFive,\tau}}\right\} -\PP\left\{\overline{E_{\circledSix,\tau}}\right\}\\
    &\geq& 1- \frac{(\tau-1)\delta}{T}  - 2\delta' -2\delta''\\
    &=& 1-\frac{\tau\delta}{T} +\left(\frac{\delta}{T} - 2\delta' -2\delta''\right)\\
    &=& 1-\frac{\tau\delta}{T},
\end{eqnarray*}
where we have selected $\delta' = \delta'' = \frac{\delta}{4T}$. The event $E_{\tau-1}$ implies for any $t \leq \tau$
\begin{equation*}
    \Delta_t \leq \gamma\sum^{t-1}_{s=0}\|\nabla F(x_s)\| + \Delta_{t} \leq 2\Delta_1.
\end{equation*}
Therefore, we have $\|\nabla F(x_t)\| \leq \sqrt{2L\Delta_t} \leq 2\sqrt{L\Delta_1}\leq \frac{\lambda}{2}$. Assuming $E_{\tau-1}\cap E_{1,\tau} \cap E_{2,\tau} \cap E_{3,\tau} \cap E_{4,\tau}$ happens, Lemma~\ref{lem:descent_lemma_clipped_sgd} implies
\begin{eqnarray*}
    \gamma\sum^{\tau}_{t=0}\|\nabla F(x_t)\| + \Delta_{\tau+1} &\leq& \Delta_1 + \frac{3\gamma}{\alpha}\sqrt{L\Delta_1} + \frac{\gamma^2 L (\tau+1)}{2}\\
    && + 2\gamma\alpha\sum^{\tau}_{t=1}\left\|\sum^t_{j=1}(1-\alpha)^{t-j} \theta_j\right\| + 2\gamma \sum^{\tau}_{t=1}\left\|\sum^t_{j=1}(1-\alpha)^{t-j+1}\omega_j\right\|\\
    &\leq& \Delta_1 + \frac{3\gamma}{\alpha}\sqrt{L\Delta_1} + \frac{\gamma^2 L (\tau+1)}{2}\\
    && + 44\gamma\alpha(\tau +1)\lambda\log\frac{8T}{\delta} + 118\gamma(\tau+1) \lambda_h\log\frac{8T}{\delta}.
\end{eqnarray*}

By setting clipping levels \eqref{eq:momentum_clipping_level} and stepsize \eqref{eq:stepsize_hp_part}, we have 
\begin{equation*}
    \gamma\sum^{\tau}_{t=0}\|\nabla F(x_t)\| + \Delta_{\tau+1} \leq \Delta_1 +\frac{\Delta_1}{4} +\frac{\Delta_1}{4} +\frac{\Delta_1}{4} +\frac{\Delta_1}{4} = 2\Delta_1.
\end{equation*}
Therefore, we obtain 
\begin{eqnarray*}
    \PP\left\{G_{\tau}\right\} &=& \PP\left\{E_{\tau}\cap E_{1,\tau}\cap E_{2,\tau} \cap E_{3,\tau} \cap E_{4, \tau}\right\}\\
    &=& \PP\left\{E_{\tau-1}\cap E_{1,\tau}\cap E_{2,\tau} \cap E_{3,\tau} \cap E_{4, \tau}\right\} \\
    &\geq& 1-\frac{\tau\delta}{T}.
\end{eqnarray*}
Thus, we have 
\begin{equation*}
    \PP\left\{E_{T}\right\} \geq \PP\left\{G_{T}\right\} \geq 1-\delta,
\end{equation*}
or equivalently with probability at least $1-\delta$ 
\begin{equation*}
    \gamma\sum^{T-1}_{t=0}\|\nabla F(x_t)\| + \Delta_{T} \leq 2\Delta_1~~\Rightarrow~~ \frac{1}{T}\sum^{T-1}_{t=0}\|\nabla F(x_t)\| \leq \frac{2\Delta_1}{\gamma T}.
\end{equation*}
Pugging \eqref{eq:stepsize_hp_part} into the inequality above, we have 
\begin{eqnarray*}
    &&\frac{1}{T}\sum^{T-1}_{t=0}\|\nabla F(x_t)\| \leq \frac{2\Delta_1}{\gamma T}\\
    &&\quad\quad = \mathcal{O}\left(\max\left\{\sqrt{\frac{L\Delta_1}{T}}, \frac{\sqrt{L\Delta_1}}{\alpha T}, \alpha \sqrt{L\Delta_1}\log \frac{T}{\delta}, \sigma\alpha^{\frac{p-1}{p}}\log \frac{T}{\delta}, \sqrt{\frac{\Delta_1(L+\sigma_h)\log\frac{T}{\delta}}{T\alpha^{\nicefrac{1}{p}}}}\right\}\right).
\end{eqnarray*}

Selecting $\alpha = \nicefrac{1}{T^{-\nicefrac{p}{2p-1}}}$ (see~\eqref{eq:momentum_clipping_level}), we obtain
\begin{eqnarray*}
    &&\frac{1}{T}\sum^{T}_{t=1}\|\nabla F(x_t)\| \\
    &&=\mathcal{O}\left(\max\left\{\sqrt{\frac{L\Delta_1}{T}}, \frac{\sqrt{L\Delta_1}}{T^{\frac{p-1}{2p-1}}},  \frac{\sqrt{L\Delta_1}}{T^{\frac{p-1}{2p-1}}}\log \frac{T}{\delta}, \frac{\sigma}{T^{\frac{p-1}{2p-1}}}\log\frac{T}{\delta}, \sqrt{\frac{\Delta_1(L+\sigma_h)\log\frac{T}{\delta}}{T^{\nicefrac{2p-2}{2p-1}}}}\right\}\right)\\
    && = \mathcal{O}\Bigg(\frac{\max\left\{\sqrt{\Delta_1(L+\sigma_h)}, \sigma\right\}}{T^{\frac{p-1}{2p-1}}}\log \frac{T}{\delta} \Bigg).
\end{eqnarray*}

\textbf{Discussion.} According to Theorem~\ref{thm:main_theorem_hp}, the sample complexity of Algorithm~\ref{alg:NSGD_SOM_clipped} is 
\begin{equation*}
    \widetilde\cO\left(\left(\frac{\max\left\{\sqrt{\Delta_1(L+\sigma_h)}, \sigma\right\}}{\varepsilon}\right)^{\frac{2p-1}{p-1}}\right),
\end{equation*}
which has the same dependence on $\varepsilon$ as lower bounds do (see Theorem~\ref{thm:lower_bounds}). However, the obtained upper bound has worst dependence on parameters $L,\sigma,\sigma_h, \Delta_1$ than lower bounds have. We suppose that this caused by the analysis technique, since Algorithm~\ref{alg:NSGD_SOM} has better convergence rate in terms of parameter dependence. This observation raises an interesting question: Is it possible to conduct high-probability analysis for Algorithm~\ref{alg:NSGD_SOM}, or equivalently, Algorithm~\ref{alg:NSGD_SOM_clipped} without clipping?


\end{proof}

\newpage
\section{Experimental Results}

We solve a simple quadratic problem, $\frac{1}{2}\|x\|^2$, in dimension $d=10$ , with synthetic noise sampled from a two-sided Pareto distribution. Recall that two-sided Pareto distribution with tail index $\bar p$ satisfies Assumption~\ref{ass:p_BCM} assumption with $p < \bar p$. We consider different values of the tail index $\bar p$ ranging from $1.1$ to $2.0$. We start each algorithm from the same point sampled uniformly at random from the standard normal distribution. Our empirical analysis consists of three experimental settings:

\begin{figure}[ht]
    \begin{minipage}[t]{0.48\linewidth}
        \centering
        \includegraphics[width=\linewidth]{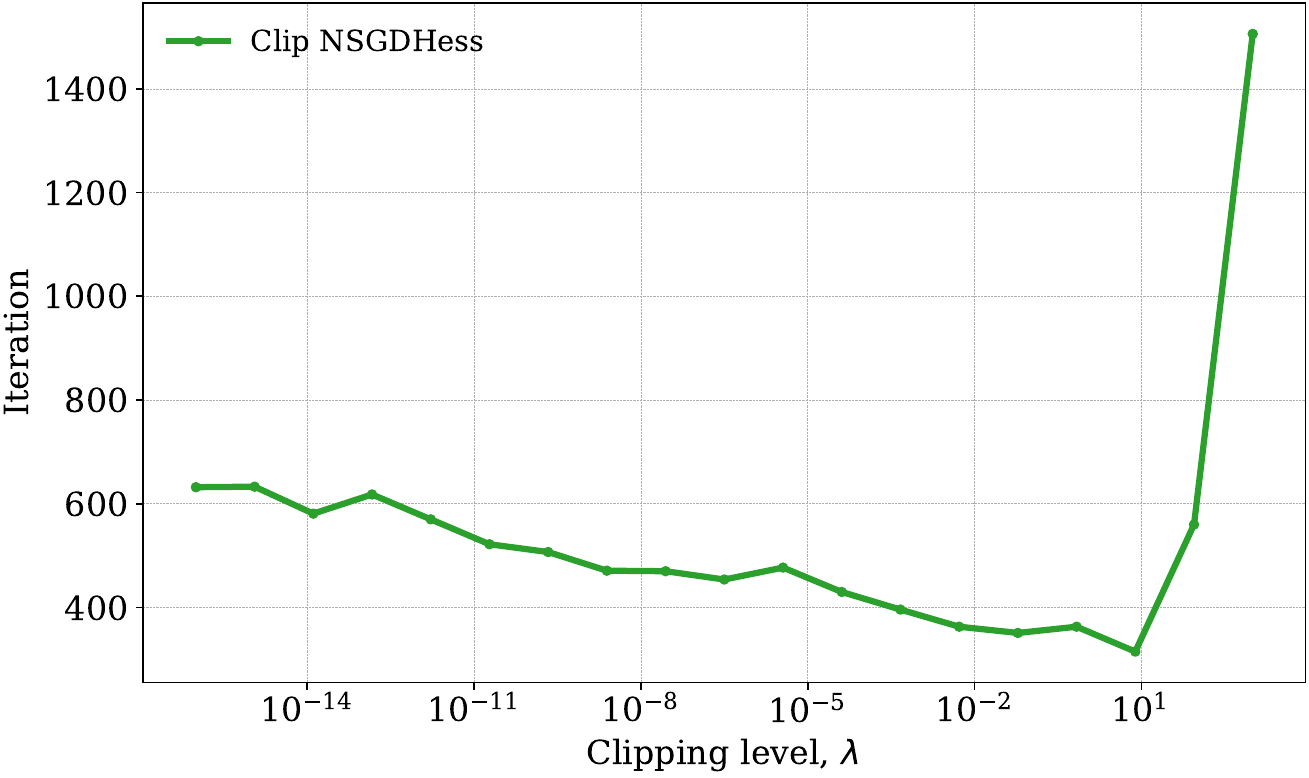}
        \caption{Effect of Hessian Clipping Level $\lambda_h = \lambda$ on the Iteration Complexity. The plot shows the number of iterations required for {Clip NSGDHess} to find a point with $\norm{\nabla F(x)} \leq 3/2$. For extremely small and large values of $\lambda,$ more iterations are needed. The recommended value for this task is $\lambda_h = 10.$
        }
        \label{fig:sensitivity_clipping}
    \end{minipage}%
    \hfill
    \begin{minipage}[t]{0.48\linewidth}
        \centering
        \includegraphics[width=\linewidth]{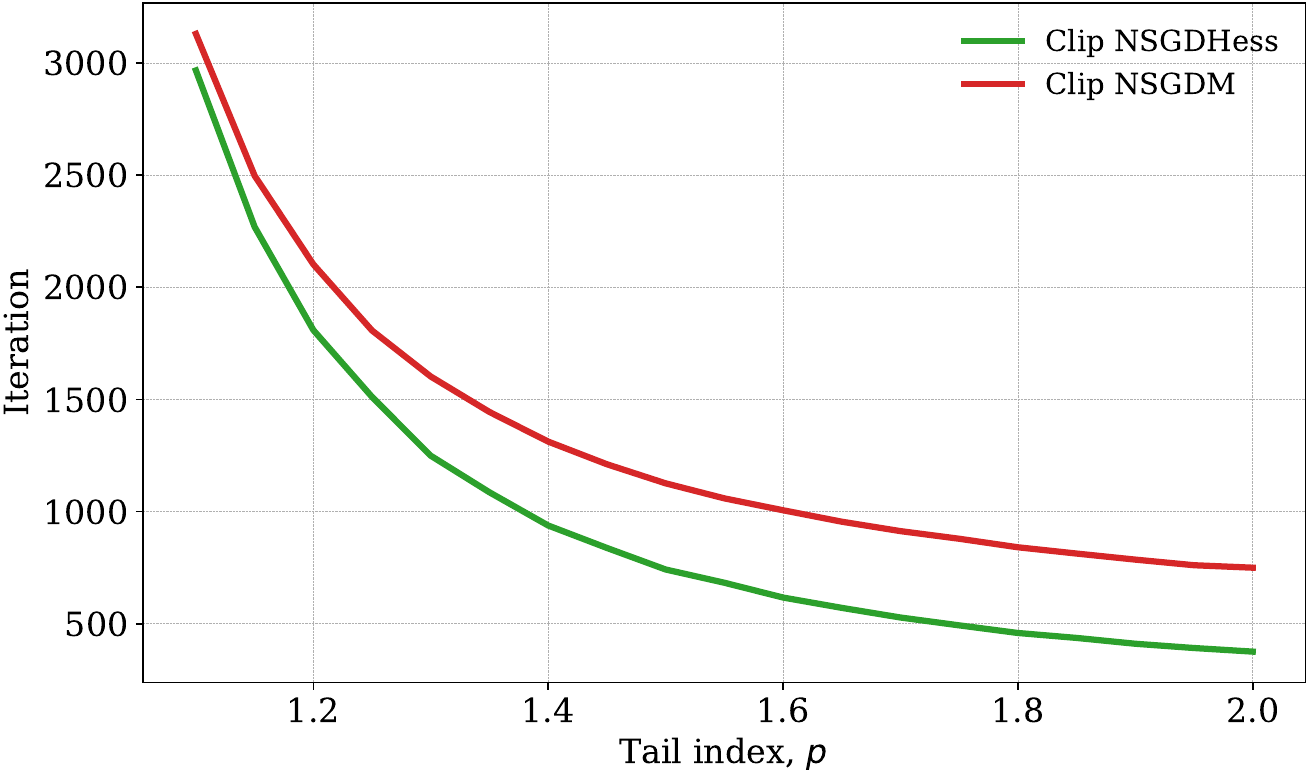}
        \caption{Number of iterations needed for Clip NSGDMHess and  Clip NSGDM under Varying Tail Index to find a point with $\norm{\nabla F(x)} \leq 3/2$ starting with the same initial point. The performance of both algorithms decreases gradually with the decrease of the tail index. The iteration complexity of the second-order algorithm, Clip NSGDMHess is uniformly better for all values of $p \in [1.1, 2].$}
        \label{fig:complexity_comparison}
    \end{minipage}
\end{figure}

\paragraph{1. Effect of Clipping Level on Iteration Complexity.} 
We run Algorithm~\ref{alg:NSGD_SOM_clipped} (Clip NSGDHess) with $\lambda_h = \lambda$ until it reached the target accuracy $\|\nabla F(x)\| \leq 3/2$. The stepsize and momentum parameters are set to $\gamma = 0.01$ and $\alpha = 0.2$. Clipping levels are varied from $10^{-16}$ to $10^{3}$ in multiplicative steps of $10$. In Figure~\ref{fig:sensitivity_clipping}, we observe that intermediate clipping ($\lambda=10$) yields the lowest iteration complexity, also refer to Table~\ref{tab:iter_comple_lambda_h} for precise numerical values. Extremely small or large values lead to slower convergence. The proposed algorithm tolerates very small values of $\lambda$, but very large values (e.g., $\lambda > 10^3$) result in very slow convergence. This empirically confirms the need for Hessian clipping to stabilize convergence and achieve fast convergence.

\begin{table}[h]
\centering
\begin{tabular}{c|ccccc}
 Clipping Level, $\lambda$ & $10^{-16}$ & $10^{-8}$ & $10^{0}$ & $10^{2}$ & $10^{3}$ \\
\hline
Iterations & 632& 471& 315& 560& 1506 \\
\end{tabular}
\vspace{0.2cm}
\caption{Iteration complexity depending on the clipping level $\lambda = \lambda_h$ for Algorithm~\ref{alg:NSGD_SOM_clipped} (Clip NSGDHess), cf. Figure~\ref{fig:complexity_comparison}.
}
\label{tab:iter_comple_lambda_h}
\end{table}

\paragraph{2. Comparison with Clip NSGDM under Varying Tail Index.}
We compare Clip NSGDHess with Clip NSGDM (which was proposed in~\cite{cutkosky2021high}). For Clip NSGDM, we used the theoretical parameter choices: stepsize $\gamma = T^{-\frac{2p-1}{3p-2}}$ and momentum $\alpha = T^{-\frac{p}{3p-2}}$. For both methods we fixed $T=4000$, $\lambda=0.5$, and $\bar{\lambda}_h=0.05$. These experiments were conducted to complement Figure~\ref{fig:complexities}, which presents the theoretical iteration complexities. As shown in Figure~\ref{fig:complexity_comparison}, the empirical results align with the theory: Clip NSGDHess consistently outperforms Clip NSGDM, and the performance gap becomes larger as $p$ increases. This provides strong empirical support for our theoretical findings.

\paragraph{3. Sensitivity to Gradient Clipping for Fixed Hessian Clipping.}
In this set of experiments, we test Clip NSGDHess (Algorithm~\ref{alg:NSGD_SOM_clipped}). Our goal is to study how iteration complexity depends on the choice of the gradient clipping level $\lambda$ when the Hessian clipping level is fixed. Specifically, we test three different values of the Hessian clipping level, and vary the gradient clipping threshold $\lambda$:
\begin{itemize}
\item For $\bar{\lambda}_h = 0.01$: $\lambda$ ranged from $10^{-4}$ to $10^{2}$ (step $10$).
\item For $\bar{\lambda}_h = 1.0$: $\lambda$ ranged from $10^{-2}$ to $10^{4}$ (step $10$).
\item For $\bar{\lambda}_h = 10$: $\lambda$ ranged from $10^{0}$ to $10^{6}$ (step $10$).
\end{itemize}
The results are summarized in Figure~\ref{fig:dmdskkcs}.

\begin{figure}[ht]
    \begin{minipage}[t]{0.48\linewidth}
        \centering
        \includegraphics[width=\linewidth]{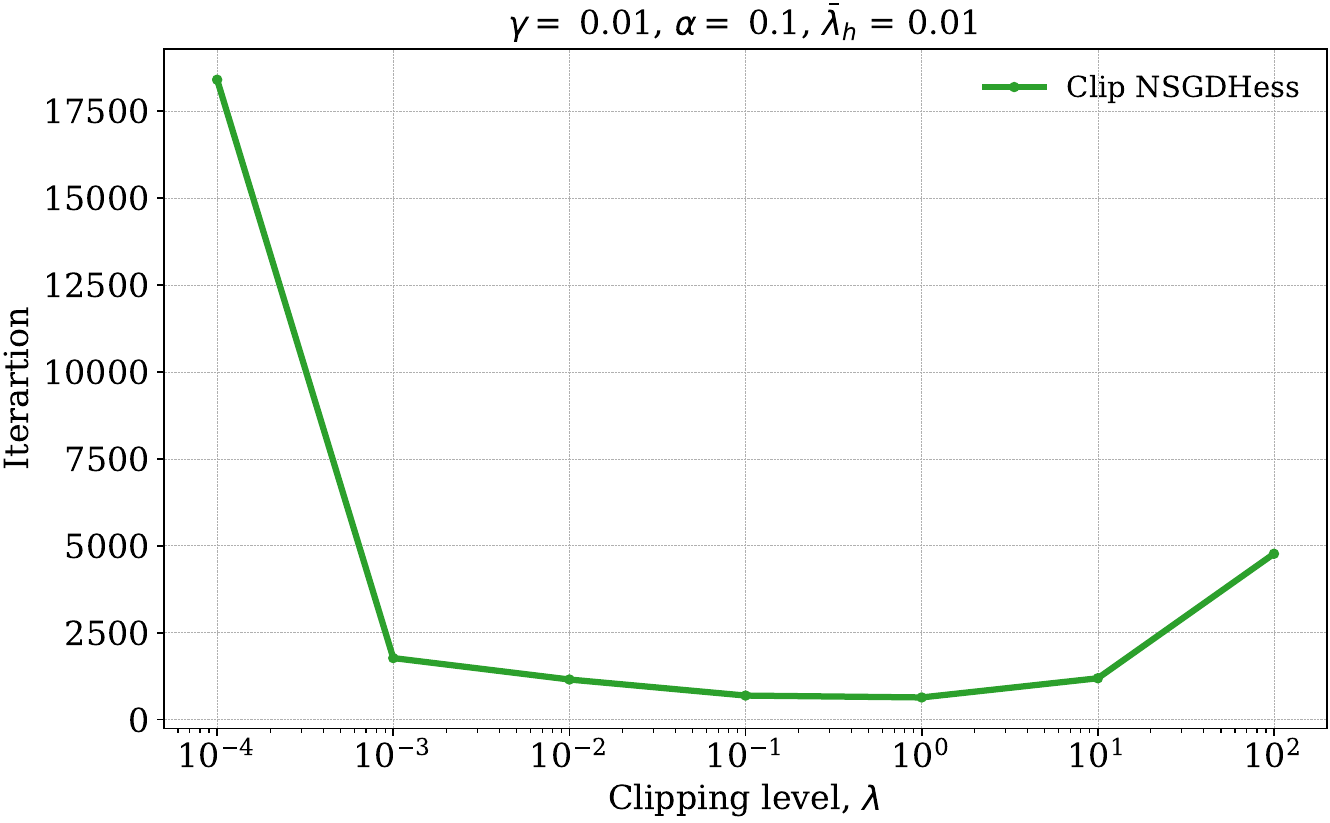}
    \end{minipage}%
    \hfill
    \begin{minipage}[t]{0.48\linewidth}
        \centering
        \includegraphics[width=\linewidth]{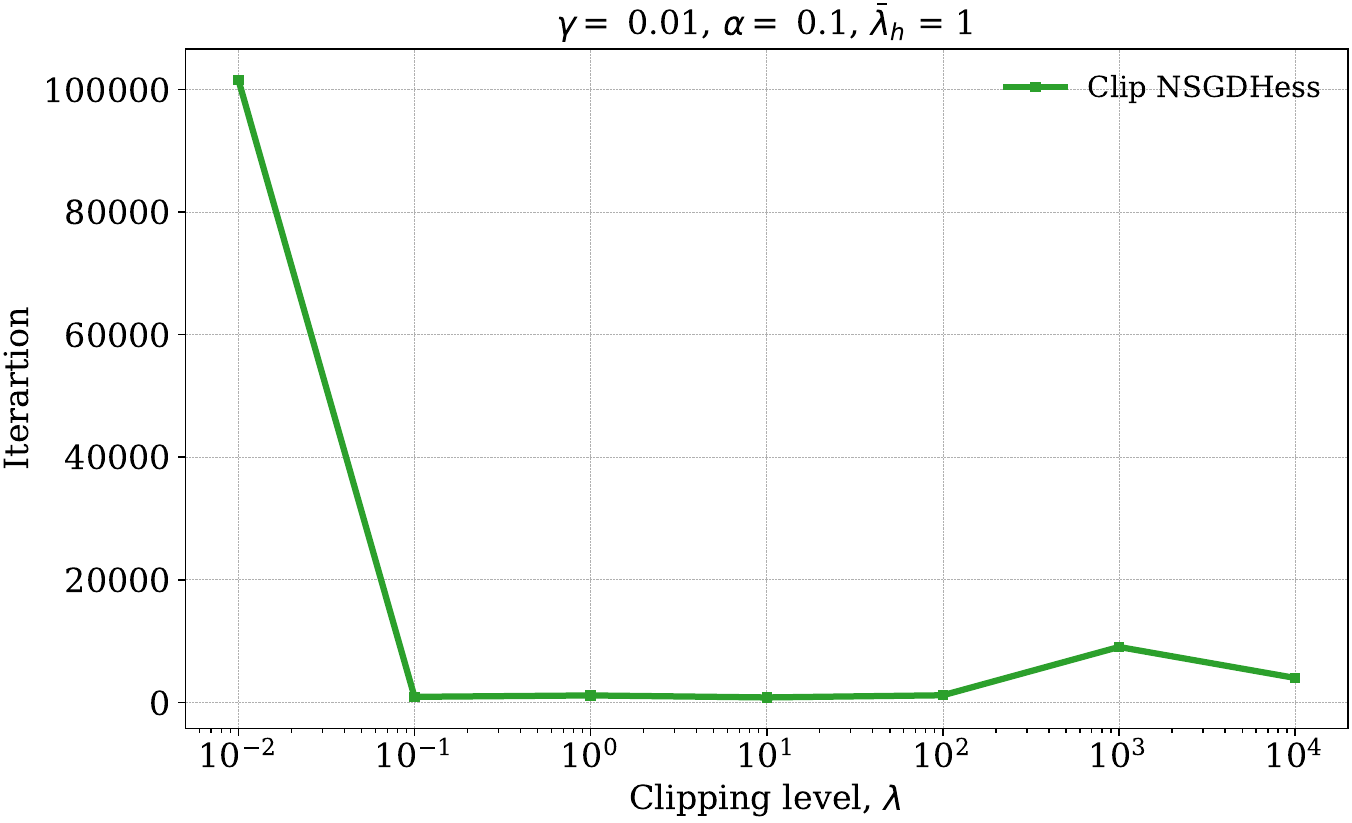}
    \end{minipage}
    \hfill
    \begin{center}
    \begin{minipage}[h]{0.48\linewidth}
        \centering
        \includegraphics[width=\linewidth]{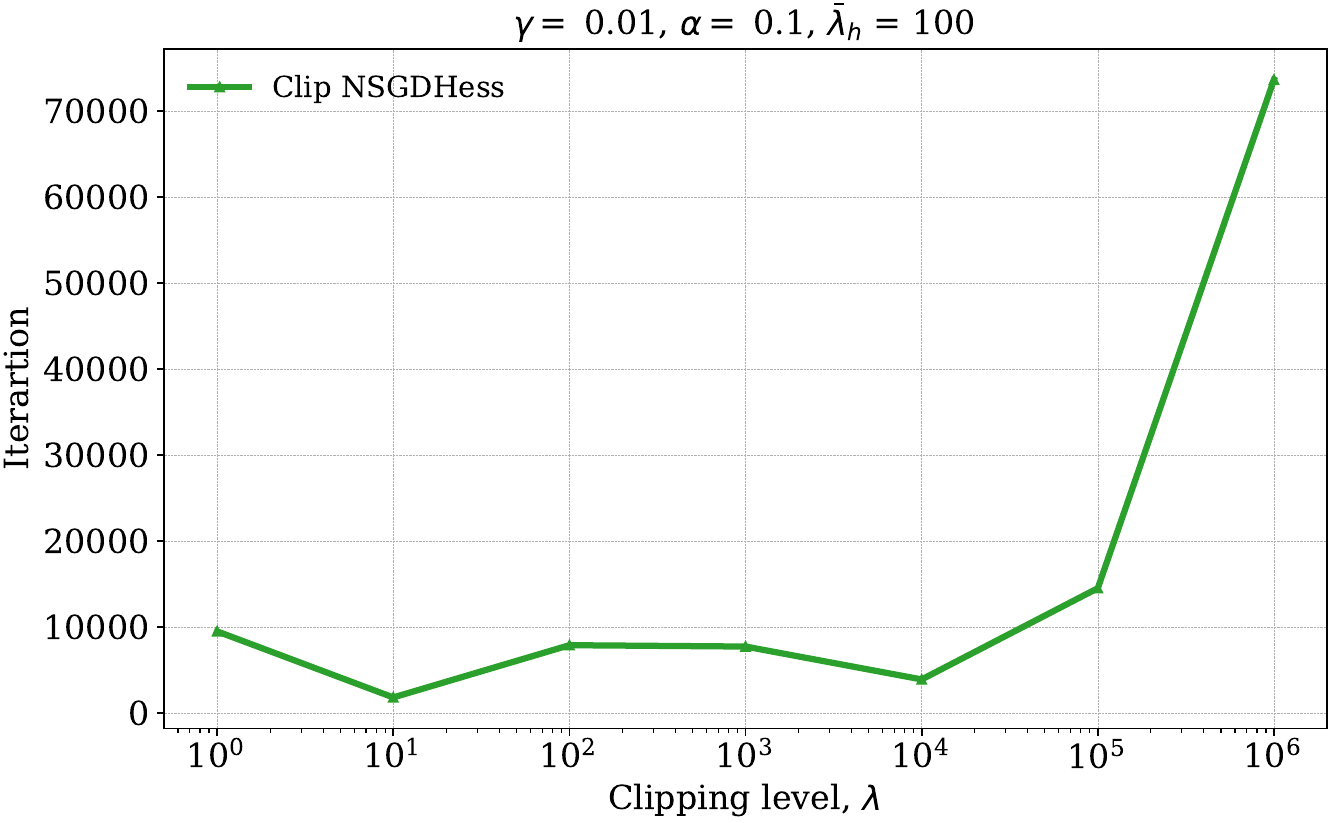}
    \end{minipage}
    \end{center}
    \caption{Iteration complexity of Clip NSGDHess (Algorithm~\ref{alg:NSGD_SOM_clipped}) depending on gradient clipping for the three different fixed values of Hessian clipping $\bar \lambda_h \in \{0.01, 1, 100\}$.}
    \label{fig:dmdskkcs}
\end{figure}

In all cases, the algorithm is tested until $\|\nabla F(x)\| \leq \tfrac{1}{2}$. We observe that intermediate values of $\lambda$ yield the best performance across all three regimes (small, medium, and large Hessian clipping levels). As shown in Figure~\ref{fig:dmdskkcs}, both extremely small and extremely large values of $\lambda$ lead to higher iteration complexity, while moderate choices minimize it. The optimal values are:
$$\lambda = 1.0 \quad \text{for} \quad \bar{\lambda}_h = 0.01, \qquad
\lambda = 0.1 \quad \text{for} \quad \bar{\lambda}_h = 1.0, \qquad
\lambda = 10.0 \quad \text{for} \quad \bar{\lambda}_h = 100.0.$$

These findings confirm that moderate clipping levels are crucial for achieving the best performance aligning with our theoretical findings in Theorem~\ref{thm:main_theorem_hp}.

\end{document}